\documentclass[a4paper]{amsart}

\usepackage[utf8]{inputenc}
\usepackage[english]{babel}
\usepackage{stmaryrd,mathtools,commath,mathrsfs}
\usepackage{bbm} %use \mathbbm for double-struck letters
\usepackage{amssymb,amsthm}
\numberwithin{equation}{section}
\usepackage{float}
\usepackage{graphicx}
\usepackage{yfonts,verbatim,calrsfs,enumitem}
\usepackage{comment}
\usepackage{todonotes}
\usepackage{scalerel} %This is for \scaleto{}, which sets font size of sub-sup-scripts

\usepackage{amsmath}
\expandafter\let\expandafter\oldproof\csname\string\proof\endcsname
\let\oldendproof\endproof
\renewenvironment{proof}[1][\proofname]{%
  \oldproof[#1]%
}{\oldendproof}
\usepackage{xcolor}
\definecolor{cornellred}{rgb}{0.7,0.11,0.11}

\usepackage{cite}
\usepackage[colorlinks=false, allcolors=cornellred, pdfpagelabels]{hyperref}

\allowdisplaybreaks

\theoremstyle{plain}
\newtheorem{theorem}{Theorem}[section]
\newtheorem{lemma}[theorem]{Lemma}
\newtheorem{proposition}[theorem]{Proposition}
\newtheorem{corollary}[theorem]{Corollary}

\theoremstyle{definition}
\newtheorem{definition}[theorem]{Definition}

\newtheorem{remark}[theorem]{Remark}

\newtheorem{notation}[theorem]{Notation}

\DeclareMathAlphabet{\mathcal}{OMS}{cmsy}{m}{n}

\newcommand{\Cs}{\ensuremath{\mathrm{C}^\ast}}
\DeclareMathOperator{\rt}{\mathsf{rt}}

\DeclareMathOperator{\1}{\mathbf{1}}

\DeclareMathOperator{\id}{\mathrm{id}}

\DeclareMathOperator{\Aut}{\mathrm{Aut}}

\DeclareMathOperator{\ev}{\mathrm{ev}}
\DeclareMathOperator{\Ad}{\mathrm{Ad}}
\DeclareMathOperator{\N}{\mathbb{N}}
\DeclareMathOperator{\Z}{\mathbb{Z}}

\DeclareMathOperator{\M}{\mathcal{M}}
\DeclareMathOperator{\U}{\mathcal{U}}
\DeclareMathOperator{\K}{\mathcal{K}}
\DeclareMathOperator{\I}{\mathcal{I}}
\DeclareMathOperator{\C}{\mathcal{C}}
\DeclareMathOperator{\B}{\mathcal{B}}

\DeclareMathOperator{\Hil}{\mathcal{H}}
\DeclareMathOperator{\O2}{\mathcal{O}_2}
\DeclareMathOperator{\Oinf}{\mathcal{O}_{\infty}}
\DeclareMathOperator{\e}{\varepsilon}
\DeclareMathOperator{\f}{\varphi}

\DeclareMathOperator{\p}{\mathfrak{p}}

\DeclareMathOperator{\bbu}{\mathbbm{u}}

\DeclareMathOperator{\Prim}{\mathrm{Prim}}

\usepackage{tikz}
\usetikzlibrary{arrows,matrix,positioning,automata}
\usepackage{quiver}
\usepackage{pgfplots}
\pgfplotsset{compat=1.18}

\setlist[enumerate,itemize]{noitemsep,nolistsep}

\begin{document}

\title[Continuous actions on primitive ideal spaces lift to \Cs-actions]{Continuous actions on primitive ideal spaces lift to \Cs-actions}
%Alternative: A dynamical version of Kirchberg's $\O2$-stable classification.

%    author one information
\author{Matteo Pagliero}
\address{Department of Mathematics, KU Leuven, Celestijnenlaan 200b, box 2400, B-3001 Leuven, Belgium}
\curraddr{}
\email{matteo.pagliero@kuleuven.be}
\thanks{}

%    author two information
%\author{}
%\address{Department of Mathematics, KU Leuven, Celestijnenlaan 200b, box 2400, B-3001 Leuven, Belgium}
%\curraddr{}
%\email{}
%\thanks{}



\date{}

\dedicatory{}

\begin{abstract}
We prove that for any second-countable, locally compact group $G$, any continuous $G$-action on the primitive ideal space of a separable, nuclear \Cs-algebra $B$ such that $B \cong B\otimes\O2\otimes\K$ is induced by an action on $B$.\
As a direct consequence, we establish that every continuous action on the primitive ideal space of a separable, nuclear \Cs-algebra is induced by an action on a \Cs-algebra with the same primitive ideal space.\
Moreover, we discuss an application to the classification of equivariantly $\O2$-stable actions.
\end{abstract}
\maketitle

\tableofcontents

\section*{Introduction}
\renewcommand\thetheorem{\rm{\Alph{theorem}}}

In this article, we investigate the following natural question:
Does a continuous group action on a primitive ideal space lift to an action on a \Cs-algebra?\
Here, the set of primitive ideals of a \Cs-algebra, i.e., those that can be realised as kernels of non-zero irreducible representations, is endowed with the Jacobson (or hull-kernel) topology.\
Given any (point-norm continuous) group action $\beta:G\curvearrowright B$ of a locally compact group on a \Cs-algebra, one can induce an algebraic action $\beta^{\sharp}:G\curvearrowright\Prim(B)$ as $\beta^{\sharp}_g(I) = \ker(\pi \circ \beta_{g^-1})$ for all $g\in G$ and irreducible non-zero representations $\pi$ of $B$.\
It is a well-known fact that the induced action $\beta^{\sharp}$ is continuous with respect to the Jacobson topology of $\Prim(B)$ (see Remark \ref{rem:continuity induced action}), and thus the question above is reasonable to begin with.\
Note that $\Prim(B\otimes\O2\otimes\K)$ is homeomorphic to $\Prim(B)$ for any \Cs-algebra $B$ because the Cuntz algebra $\O2$ and the \Cs-algebra of compact operators $\K$ are simple.\
Therefore, one can assume $B\cong B\otimes\O2\otimes\K$ without loss of generality for the problem under consideration.\
We are able to answer the question above in the positive as follows.

\begin{theorem}\label{thm-intro:realisation}
Let $G$ be a second-countable, locally compact group, and $B$ a separable, nuclear \Cs-algebra such that $B\cong B\otimes\O2\otimes\K$.\
For every continuous action $\gamma:G\curvearrowright\Prim(B)$, there exists an action $\beta:G\curvearrowright B$ such that $\beta^{\sharp}=\gamma$.
\end{theorem}

Kirchberg's classification theorem for $\O2$-stable \Cs-algebras is a key result in the classification of non-simple, strongly purely infinite, separable, nuclear \Cs-algebras.\
First announced by Kirchberg in 1994, and only recently published with a new proof by Gabe \cite{Gab20}, it states that two separable, nuclear, stable (or unital) \Cs-algebras that tensorially absorb the Cuntz algebra $\O2$ are isomorphic precisely when their primitive ideal spaces are homeomorphic.\
Motivated by the techniques developed by Gabe and Szabó in \cite{GS22a,GS22b} to obtain a dynamical version of the Kirchberg--Phillips theorem \cite{Kir,KP00,Phi00}, Szabó and the author have recently generalised Kirchberg's theorem for $\O2$-stable \Cs-algebras to a classification of second-countable, locally compact group actions up to cocycle conjugacy (see \cite[Theorem 5.6]{PS23}).\
This states that two amenable, isometrically shift-absorbing, equivariantly $\O2$-stable actions $\alpha:G\curvearrowright A$ and $\beta:G\curvearrowright B$ on separable, nuclear, stable \Cs-algebras are cocycle conjugate if and only if $\alpha^{\sharp}$ is conjugate to $\beta^{\sharp}$.\
Cocycle conjugacy is weaker than conjugacy, i.e., equivariant isomorphism, and it is the right notion of isomorphism in the classification theory of locally compact group actions on \Cs-algebras (see \cite{Sza21} for a detailed outline of this categorical framework).

\begin{definition}\label{def:comorp}
Let $\alpha:G\curvearrowright A$ and $\beta:G\curvearrowright B$ be actions on \Cs-algebras.\
A \emph{cocycle conjugacy} from $(A,\alpha)$ to $(B,\beta)$ is a pair $(\f,\bbu)$ consisting of a $\ast$-isomorphism $\f:A\to B$ and a strictly continuous map $\bbu:G\to\U(\M(B))$ such that $\f \circ \alpha_g = \Ad(\bbu_g) \circ \beta_g \circ \f$ and $\bbu_{gh}=\bbu_g\beta_{g}(\bbu_h)$ for all $g,h\in G$.
\end{definition}

In the last part of the present work, we are able prove the following result using the aforementioned classification result from \cite{PS23} and Theorem \ref{thm-intro:realisation}.

\begin{theorem}\label{cor-intro}
Let $G$ be a second-countable, locally compact group, and $A$ a separable, nuclear, stable, $\O2$-stable \Cs-algebra.\
Then associating $\alpha^{\sharp}$ to each action $\alpha:G\curvearrowright A$ gives the following bijection:
\begin{equation*}
\frac{\begin{Bmatrix}\text{\rm{amenable, isometrically shift-absorbing,}}\\ \text{\rm{equivariantly $\O2$-stable actions }}G\curvearrowright A \end{Bmatrix}}{\text{\rm{cocycle conjugacy}}} \longrightarrow \frac{\begin{Bmatrix}\text{\rm{continuous actions}}\\ G\curvearrowright\Prim(A) \end{Bmatrix}}{\text{\rm{conjugacy}}}
\end{equation*}
If $G$ is moreover compact, ``cocycle conjugacy" above can be replaced with ``conjugacy".
\end{theorem}

Phillips has recently made available a joint work with Kirchberg in which they prove a partial version of Theorem \ref{thm-intro:realisation}, namely \cite[Theorem 5.9]{KP23}, under the additional assumption that the primitive ideal space is homeomorphic to the (non-Hausdorff) compactification in the sense of \cite[Definition 1.5]{KP23} of a second-countable, non-compact, locally compact space, which, in particular, forces $A$ to be prime (see \cite[Corollary 1.7]{KP23}).\
This overlaps with results in this article, and it would not be hard to adapt their proof to obtain Theorem \ref{thm-intro:realisation} in full generality, although this is not explicitly stated by the authors.\
In fact, the approach of the present work is conceptually inspired by Kirchberg--Phillips' core idea, which is outlined in an extended abstract within an Oberwolfach report \cite{Obe}.\
However, the author was not aware of their subsequent work.\
Moreover, it should be noted that this article is completely self-contained and, contrarily to Kirchberg--Phillips',  does not require any preliminary knowledge about the work of Harnisch--Kirchberg \cite{HK24}, which has only very recently been made available.\
Conversely, the present work does not provide information about what topological spaces arise as primitive ideal spaces, which is the main novelty of Harnisch and Kirchberg's work.

Section 1 and 2 can be viewed as a self-contained exposition of (part of) Harnisch--Kirchberg's work \cite{HK24}, which will be relevant in the rest of the article.\
We choose to report all proofs, and while some are just recalled from \cite{HK24}, others are new and shorter.\
An example of this is Lemma \ref{lem:aperiodic}, where we can directly show aperiodicity of the $\Z$-action, while the original proof needed to pass to the double dual of the \Cs-algebra.
The main aim of Section 3 is to present the right setting for proving our main theorem.\
In particular, we recall a construction from \cite{HK24}, and give a dynamical version of it that will be useful in the rest of this work, which we believe is of independent interest.\
In the final part of the article, Section 4, we prove the main result and the aforementioned application to the classification of equivariantly $\O2$-stable actions.

%%% % %%%

\renewcommand\thetheorem{\arabic{section}.\arabic{theorem}}

\section{Harnisch--Kirchberg systems}
The notation present in this section follows that of \cite{HK24}, with the exception of Definition \ref{def:hk-system}, where we use the name \emph{Harnisch--Kirchberg system} for a dynamical object that is called $\sigma$-modular in \cite{HK24}.\
We recapitulate a number of results from the same work, and complement them with some new useful observations.\
We always include proofs to make the present work self-contained.

In the present work, we choose by convention that $0\in\N$.

\begin{notation}[{cf.\ \cite[Section 4]{HK24}}]
Let $B$ be a \Cs-algebra, $\sigma\in\Aut(B)$ an automorphism, and $D \subseteq B$ a \Cs-subalgebra.\
For every $k\in\Z$ and $\ell \geq k$, one defines the vector space
\begin{align*}
D_{k,\ell}=\sigma^{k-1}(D)+\dots+\sigma^{\ell-1}(D).
\end{align*}
One can moreover form the following vector spaces,
\begin{equation*}
D_{-\infty,k}=\overline{\bigcup_{n\in\N}D_{k-n,k}}, \quad D_{k,\infty}=\overline{\bigcup_{n\in\N}D_{k,k+n}}, \quad D_{\sigma}=\overline{\bigcup_{n\in\N}D_{-n,n}}.
\end{equation*}
\end{notation}

\begin{lemma}[{cf.\ \cite[Lemma 4.7(i)]{HK24}}]\label{lem:properties1}
Let $B$ be a \Cs-algebra, $\sigma\in\Aut(B)$, and $D \subseteq B$ a \Cs-subalgebra.\
For every $i,k,\ell\in\Z$ with $k\leq\ell$, one has that
\begin{align*}
\sigma^i(D_{k,\ell})=D_{k+i,\ell+i},  &&&\sigma^i(D_{-\infty,\ell})=D_{-\infty,\ell+i},\\
\sigma^i(D_{k,\infty})=D_{k+i,\infty}, &&&\sigma^i(D_\sigma)=D_\sigma.
\end{align*}
\end{lemma}
\begin{proof}
The following identities are true by definition,
\begin{equation*}
\sigma^i(D_{k,\ell})=\sigma^i(\sigma^{k-1}(D)+\dots+\sigma^{\ell-1}(D))=\sigma^{k+i-1}(D)+\dots+\sigma^{\ell+i-1}(D)=D_{k+i,\ell+i}
\end{equation*}
for all $k\leq\ell$ and $i$ in $\Z$.\
This implies that the statements for $D_{-\infty,\ell}$, $D_{k,\infty}$ and $D_{\sigma}$ also hold true as they are inductive limits of $\{D_{\ell-n,\ell}\}_{n\in\N}$, $\{D_{k,k+n}\}_{n\in\N}$ and $\{D_{-n,n}\}_{n\in\N}$, respectively.
\end{proof}

\begin{lemma}[{cf.\ \cite[Lemma 4.8]{HK24}}]\label{lem:subalgebra}
Let $C,D$ be \Cs-subalgebras of a \Cs-algebra $B$ such that $CD\subseteq D$.\
Then the $\ast$-subalgebra of $B$ given by $C+D$ is a \Cs-algebra.
\end{lemma}
\begin{proof}
We know that $D$ is a closed two-sided ideal of $C+D$ as $CD\subseteq D$, and $DC=(C^*D^*)^*=(CD)^*\subseteq D^*=D$.\
Moreover, $C\cap D$ is a closed two-sided ideal of $C$.\
The map $\f: C/(C\cap D) \to (C+D)/D$ given by $\f(x+ (C\cap D)) = x+ D$ for all $x\in C$, is a well-defined $\ast$-homomorphism of $\ast$-algebras.\
Moreover, it is injective for if $x+D = y+D$ for some $x,y\in C$, we have that $(x-y)\in C\cap D$, and thus $y+(C\cap D)=y+(x-y)+(C\cap D)=x+(C\cap D)$.\
It is surjective because for any element $x+D$ of $(C+D)/D$, there exists $x'\in C$ such that $x+D=x'+D$ with $x'\in C$, and $\f(x'+(C\cap D)) = x'+D$.\
Since $C/(C\cap D)$ is a \Cs-algebra, so is $\f(C/(C\cap D))=(C+D)/D$.\
Hence, being the preimage of the quotient map $(C+D)\to(C+D)/D$, $C+D$ is norm-closed, and thus a \Cs-algebra.
\end{proof}

\begin{lemma}[{cf.\ \cite[Lemma 4.7(ii)]{HK24}}]\label{lem:properties2}
Let $B$ be a \Cs-algebra, $\sigma\in\Aut(B)$, and $D \subseteq B$ a \Cs-subalgebra such that $D\sigma(D)=\sigma(D)$.\
Then, for all $n>0$ and integers $k\leq\ell$ and $i\leq j$, one has that
\begin{equation*}
D\sigma^n(D)=\sigma^n(D), \quad \text{and}\quad D_{k,\ell}D_{i,j}\subseteq D_{\max(k,i),\max(\ell,j)}.
\end{equation*}

Moreover, $D_{k,\ell},D_{-\infty,\ell}$, and $D_{k,\infty}$ are \Cs-subalgebras of $B$, and $D_{\sigma}$ is the smallest $\sigma$-invariant \Cs-subalgebra of $B$ containing $D$.
\end{lemma}
\begin{proof}
First, observe that if $n=1$, $D\sigma^n(D)=D\sigma(D)=\sigma(D)$ by assumption.\
Assume that $D\sigma^n(D)=\sigma^n(D)$ for some $n>1$, and note that
\begin{equation*}
D\sigma^{n+1}(D) = D\sigma^n(D\sigma(D)) = D\sigma^n(D)\sigma^{n+1}(D)=\sigma^n(D\sigma(D))=\sigma^{n+1}(D).
\end{equation*}
By induction on $n\in\N$ we get that the first condition holds.\
Observe that whenever $i\leq j$ are integers, by the previous part of the proof we have that
\begin{equation*}
\sigma^i(D)\sigma^j(D)=\sigma^i(D\sigma^{j-i}(D))=\sigma^i(\sigma^{j-i}(D))= \sigma^j(D).
\end{equation*}
Then, since we know that $\sigma^j(D)\sigma^i(D)=\left(\sigma^i(D)\sigma^j(D)\right)^*=\sigma^j(D)$, this proves the second condition.

Fore the moreover part, note that each $D_{k,k}=\sigma^{k-1}(D)$, for $k\in\Z$, is a \Cs-algebra as $\sigma$ is an automorphism.\
It follows from Lemma \ref{lem:subalgebra} that $D_{k,k+1}=D_{k,k}+\sigma^{k}(D)$ is a \Cs-algebra as well because $D_{k,k}\sigma^{k}(D)=\sigma^{k-1}(D\sigma(D))\subseteq\sigma^{k}(D)$.\
Suppose that $D_{k,\ell}$ is a \Cs-algebra, where $\ell>k+1$.\
It follows from the first part of the lemma, which we have just proven, that $D_{k,\ell}\sigma^{\ell}(D)\subseteq\sigma^{\ell}(D)$, and thus $D_{k,\ell+1}=D_{k,\ell}+\sigma^{\ell}(D)$ is a \Cs-algebra by Lemma \ref{lem:subalgebra}.\
By induction on $\ell>k+1$, we have that $D_{k,\ell}$ is a \Cs-algebra for all pairs of integers $k\leq\ell$.\
As a consequence, $D_{k,\infty}$ and $D_{-\infty,\ell}$ are \Cs-algebras for all $k,\ell\in\Z$.\
Since $D_\sigma$ is generated by all $\sigma^{k}(D)$, for $k\in\Z$, it is the smallest $\sigma$-invariant \Cs-subalgebra of $B$ containing $D$.
\end{proof}

It should be noted that, throughout the article, whenever we say that $I$ is an ideal of a \Cs-algebra $A$, we assume that $I$ closed and two-sided.

\begin{definition}
Let $A$ be a \Cs-algebra.\
An ideal $I$ of $A$ is \emph{essential} if for every non-zero ideal $J\in\I(A)$, $I\cap J\neq\{0\}$.\
Equivalently, $I$ is an essential ideal of $A$ if $A\cap I^{\perp} = \{a\in A \mid ab=0,\; \text{for all}\; b\in I\} = \{0\}$.
\end{definition}

\begin{lemma}[{cf.\ \cite[Lemma 4.9+4.10]{HK24}}]\label{lem:ideals}
Let $B$ be a \Cs-algebra, $\sigma\in\Aut(B)$, and $D \subseteq B$ a \Cs-subalgebra such that $D\sigma(D)=\sigma(D)$.\
Then, for all integers  $k_1\leq k_2\leq \ell$, $D_{k_2,\ell}$ is an ideal of $D_{k_1,\ell}$, and $D_{\ell,\infty}$ is an ideal of $D_\sigma$.

If moreover $D\cap\sigma(D)=\{0\}$ and $D\cap\sigma(D)^{\perp}= \{0\}$,\footnote{Here, one should formally write $(D+\sigma(D))\cap\sigma(D)^{\perp}$. However, since the only element of $\sigma(D)$ that annihilates $\sigma(D)$ itself is $0$, we abuse notation, and write $D\cap\sigma(D)^{\perp}$ instead.} then $D_{k_2,\ell}$ is an essential ideal of $D_{k_1,\ell}$, and $D_{\ell,\infty}$ is an essential ideal of $D_\sigma$.
\end{lemma}
\begin{proof}
The first part follows from Lemma \ref{lem:properties2}.

The assumption asking that $D\cap\sigma(D)^{\perp}= \{0\}$ is equivalent to saying that $D_{2,2}=\sigma(D)$ is an essential ideal of $D_{1,2}=D+\sigma(D)$.\
Suppose that $D_{n,n}$ is an essential ideal of $D_{1,n}$.\
Let $x\in D_{1,n+1}$ with decomposition $x=y+z$, where $y\in D_{1,n}$ and $z\in D_{n+1,n+1}$, and assume that $xD_{n+1,n+1}=\{0\}$.\
Then, for every element $d\in D_{n,n}$, we have from Lemma \ref{lem:properties2} that $dx=dy+dz\in D_{n,n}+D_{n+1,n+1}=\sigma^{n-1}(D_{1,2})$.\
Since $(D+\sigma(D))\cap\sigma(D)^{\perp}= \{0\}$, we have that $dxD_{n+1,n+1}=\{0\}$ implies $dx=0$.\
Hence, $dy=-dz\in \sigma^{n-1}(D\cap\sigma(D))=\{0\}$, so that $D_{n,n}y=0$ and since we assumed that $D_{n,n}$ is essential in $D_{1,n}$, it must be $y=0$.\
Therefore, $x=z\in D_{n+1,n+1}$ with $xD_{n+1,n+1}=0$, which implies that $x=0$.\
Arguing by induction on $n\geq1$, we have shown that $D_{n,n}$ is an essential ideal of $D_{1,n}$ for all $n\geq1$.

Now, for the general case, pick integers $k_1\leq k_2\leq \ell$.\
To show that $D_{k_2,\ell}$ is an essential ideal of $D_{k_1,\ell}$, first observe that $D_{\ell,\ell}=\sigma^{k_1-1}(D_{\ell+1-k_1,\ell+1-k_1})$ is essential in $D_{k_1,\ell}=\sigma^{k_1-1}(D_{1,\ell+1-k_1})$.\
Moreover, we have already observed that $D_{k_2,\ell}$ is an ideal of $D_{k_1,\ell}$.\
Since $D_{k_2,\ell}$ contains $D_{\ell,\ell}$, it follows that it is an essential ideal of $D_{k_1,\ell}$.

By definition, $D_\sigma$ is the inductive limit \Cs-algebra of $\{D_{-n,n}\}_{n\in\N}$ with canonical inclusions.\
Hence, if $J\subseteq D_{\sigma}$ is a non-zero ideal, it is the limit of the inductive system given by $J_n=J\cap D_{-n,n}$ with restriction of canonical inclusions for all $n\in\N$.\
Fix some $\ell\in\Z$.\
One has that there exists $n_0\in\N$ with $n_0\geq\ell$, and such that $J_{n_0}\neq\{0\}$.\
Since $D_{n_0,n_0}$ is an essential ideal of $D_{-n_0,n_0}$, it follows that $J_{n_0}\cap D_{n_0,n_0}\neq\{0\}$.\
Finally, one may conclude that $D_{\ell,\infty}$ is an essential ideal of $D_\sigma$ as $J\cap D_{\ell,\infty} \supseteq J_{n_0}\cap D_{n_0,n_0}$.
\end{proof}

The following definition is introduced with the sole purpose of streamlining a number of statements later in the article.

\begin{definition}[{cf.\ \cite[Definition 4.2]{HK24}}]\label{def:hk-system}
Let $B$ be a \Cs-algebra equipped with an automorphism $\sigma\in\Aut(B)$, and $D\subseteq B$ a \Cs-subalgebra.
We say that the triple $(B,\sigma,D)$ is a \emph{Harnisch--Kirchberg system} if the following conditions hold,
\begin{enumerate}[label=\textup{(\roman*)},leftmargin=*]
\item $D\sigma(D)=\sigma(D)$,
\item $D\cap\sigma(D) = \{0\}$,
\item $D\cap\sigma(D)^{\perp}= \{0\}$.
\end{enumerate}
\end{definition}

Note that, if $(B,\sigma,D)$ is a triple with the properties of Definition \ref{def:hk-system}, then Harnisch and Kirchberg say that $D$ is $\sigma$-modular (see \cite[Definition 4.2]{HK24}).\
We avoid using this terminology.

\begin{lemma}[{cf.\ \cite[Lemma 4.11]{HK24}}]\label{lem:independence}
Let $(B,\sigma,D)$ be a Harnisch--Kirchberg system.\
Then, for all integers $k_1\leq k_2< k_3$, we have that $D_{k_1,k_2}\cap D_{k_2+1,k_3}=\{0\}$.
\end{lemma}
\begin{proof}
Observe that it is enough to show that $D_{k_2,k_2}\cap D_{k_2+1,k_3}=\{0\}$ for if $D_{k_1,k_2}\cap D_{k_2+1,k_3}$ is a non-zero ideal of $D_{k_1,k_2}$, then in particular $D_{k_2,k_2}\cap D_{k_2+1,k_3}\neq\{0\}$ as by Lemma \ref{lem:ideals} $D_{k_2,k_2}$ is an essential ideal of $D_{k_1,k_2}$.\
Hence, after applying $\sigma$ a suitable number of times, it is sufficient to show that $D \cap D_{2,n}=\{0\}$ for all $n\geq2$.

By assumption, $D \cap D_{2,2}=D\cap\sigma(D)=\{0\}$.\
Suppose that $D \cap D_{2,n}=\{0\}$ for some $n>2$.\
Assume, towards a contradiction, that there exists a non-zero element $d\in D\cap D_{2,n+1}$, and find $v\in D_{2,n}$ and $w\in D_{n+1,n+1}$ such that $d=v+w$.\
It follows that $w\neq0$ for if $w=0$, then $d=v \in D \cap D_{2,n}=\{0\}$.\
Now, we have shown that $w=d-v\in D_{1,n}\cap D_{n+1,n+1}$ is a non-zero element, and therefore $D_{1,n}\cap D_{n+1,n+1}$ is a non-zero ideal of $D_{1,n}$.\
Since $D_{n,n}$ is an essential ideal of $D_{1,n}$ by Lemma \ref{lem:ideals}, we have that $D_{n,n}\cap D_{1,n}\cap D_{n+1,n+1}\neq0$, and thus $\sigma^n(D\cap\sigma(D))\neq0$, a contradiction.\
We argue by induction on $n$ that $D \cap D_{2,n}=\{0\}$ for all $n\geq2$, which finishes the proof.
\end{proof}

\begin{lemma}\label{lem:approxunit}
Let $(B,\sigma,D)$ be a Harnisch--Kirchberg system.\
For any given integer $k\in\Z$, an approximate unit of $D_{k,k}$ is also an approximate unit of $D_{k,\infty}$.
\end{lemma}
\begin{proof}
By assumption $D\sigma(D)=\sigma(D)$, hence $D_{k,k}D_{k+1,k+1}=D_{k+1,k+1}$.\
Assume that $\{u_\lambda\}_{\lambda} \subseteq D_{k,k}$ is an approximate unit for $D_{k,k}$.\
We have that for any $x\in D_{k+1,k+1}$, there exist $y\in D_{k,k}$ and $z\in D_{k+1,k+1}$ such that $x=yz$.\
Therefore, $\lim_{\lambda}u_\lambda x=\lim_{\lambda}(u_\lambda y)z = yz=x$.\
Assume now that $\{u_\lambda\}_{\lambda} \subseteq D_{k,k}$ is an approximate unit for $D_{k,\ell}$, for some $\ell>k$.\
It follows in particular that $\{u_\lambda\}_{\lambda}$ is an approximate unit for $D_{\ell,\ell}$, and, from the previous part of the proof, of $D_{\ell+1,\ell+1}$.\
This implies that $\{u_\lambda\}_{\lambda}$ is an approximate unit for $D_{k,\ell+1}$.\
By induction on $\ell\geq k$, $\{u_\lambda\}_{\lambda}$ is an approximate unit for $D_{k,\ell}$ for all $\ell\geq k$, and thus for $D_{k,\infty}$.
\end{proof}

\begin{notation}
Given a \Cs-algebra $A$ and an automorphism $\sigma$ of $A$, we denote by $A\rtimes_{\sigma}\Z$ the full crossed product \Cs-algebra of $A$ by the $\Z$-action induced by $\sigma$.\
For a definition of full crossed products by discrete groups see, e.g., \cite[Definition 4.1.2]{BO08}.\
Note that here the reduced crossed product is the same as the full one because $\Z$ is amenable (see, e.g., \cite[Theorem 4.2.6]{BO08}).\
It is important to keep in mind, as it will be used later in the article, that if $I$ is a $\sigma$-invariant ideal of $A$, then $I\rtimes_{\sigma}\Z$ can be naturally identified with an ideal of $A\rtimes_{\sigma}\Z$.
\end{notation}

We clarify the terminology used in the following lemma.\
The \Cs-algebra of compact operators on a Hilbert space $\Hil$ is denoted by $\K(\Hil)$, and when $\Hil$ is separable and infinite-dimensional one usually uses the shorthand notation $\K$.\
Moreover, a \Cs-algebra $A$ is said to be stable if $A\cong A\otimes\K$.

\begin{proposition}\label{prop:stability}
Let $(B,\sigma,D)$ be a Harnisch--Kirchberg system, where $D$ is $\sigma$-unital.\
If $D$ is stable, then so are $D_{\sigma}$ and $D_{\sigma}\rtimes_{\sigma}\Z$.
\end{proposition}
\begin{proof}
By combining Lemma \ref{lem:approxunit} and \cite[Proposition 4.4]{HR98} we obtain that $D_{1,\infty}$ is stable.\
Hence, $D_\sigma$ can be expressed as the inductive limit of stable, $\sigma$-unital \Cs-algebras, thus implying with \cite[Corollary 4.1]{HR98} that it is stable.\
Moreover, it follows from \cite[Corollary 4.5]{HR98} that $D_{\sigma}\rtimes_{\sigma}\Z$ is stable as well.
%Note that countable inductive limit of sigma unital C*-algebras is sigma-unital
\end{proof}

\section{Ideal structure of Harnisch--Kirchberg systems}
In this section, we investigate the ideal structure of Harnisch--Kirchberg systems.\
The main result of this section, Theorem \ref{thm:orderisomorphism}, first appeared in \cite{HK24}.\
Nevertheless, here we present a new proof, which differs from previous work in that it does not involve the use of the double dual of a \Cs-algebra to obtain Lemma \ref{lem:aperiodic}, exploiting instead work of Kwa{\'s}niewski and Meyer \cite{KM18}, which yields a shorter, and more direct, proof.

Let us first recall some well-known facts about the ideal structure of \Cs-algebras.

\begin{definition}
The set $\I(A)$ of all closed, two-sided ideals of $A$ ordered by inclusion is a lattice, where sum and intersection of ideals provide the join and meet operations, respectively.

The set $\Prim(A)$ of all primitive ideals of $A$, i.e., the elements $I\in\I(A)$ of the form $\ker(\pi)=I$ for non-zero irreducible representations $\pi$ of $A$, is endowed with the Jacobson (or hull-kernel) topology.\
Recall that this topology is defined by saying that the closure of a subset $U\subseteq\Prim(A)$ is
\begin{equation*}
\overline{U}=\bigg\{J\in\Prim(A) \mid J \supseteq \bigcap_{I\in U} I \bigg\}.
\end{equation*}
\end{definition}

\begin{definition}
    Let $B$ be a \Cs-algebra, $A\subseteq B$ a \Cs-subalgebra, and $\mathcal{L}\subseteq\I(B)$ a non-empty subset of the ideal lattice of $B$.\
    The \Cs-algebra $A$ is said to \textit{detect ideals in $\mathcal{L}$} if every non-zero ideal in $\mathcal{L}$ has non-zero intersection with $A$.\footnote{Note that if $\mathcal{L}$ only contains the zero ideal, this is a vacuus condition.}
    If $A$ detects ideals in $\I(B)$, one says that \emph{$A$ detects ideals in $B$}.\footnote{In the literature, an inclusion $A\subseteq B$ where $A$ detects ideals in $B$ may be called essential, or satisfying the ideal intersection property.\ We follow the notation used in \cite[Definition 2.2]{KM18}}.
\end{definition}

Detection of ideals is ultimately a representation theoretic property.\
Indeed, it is well-known that a \Cs-subalgebra $A$ detects ideals in $B$ if and only if every $\ast$-representation of $B$ with faithful restriction to $A$ was already faithful on $B$.\
In the following lemma, we generalise this fact.

\begin{lemma}\label{lem:representation}
Let $B$ be a \Cs-algebra, $A\subseteq B$ a \Cs-subalgebra, and $\mathcal{L}\subseteq\I(B)$ be a subset.\
$A$ detects ideals in $\mathcal{L}$ if and only if every $\ast$-representation of $B$ which is faithful on $A$ and whose kernel is in $\mathcal{L}$ is in fact faithful on $B$.
\end{lemma}
\begin{proof}
When $\mathcal{L}$ only contains the trivial ideal, the equivalence is a tautology.\
Hence, assume that $\mathcal{L}$ contains at least a non-zero ideal.

Suppose that $A$ detects ideals in $\mathcal{L}$, and pick a $\ast$-representation $\rho$ of $B$ whose restriction to $A$ is faithful and $\ker(\rho)\in\mathcal{L}$.\
Since $\ker(\rho)\cap A=\{0\}$, it must be that $\ker(\rho)=\{0\}$.

For the converse, pick a non-zero ideal $I\in\mathcal{L}$, and assume that every $\ast$-representation of $B$ that is faithful on $A$ and whose kernel is in $\mathcal{L}$ is faithful on $B$.\
Choose a faithful representation of $B/I$, and denote by $\rho$ its composition with the quotient map $B\to B/I$.\
We have that $\rho$ is not faithful and $\ker(\rho)=I\in\mathcal{L}$.\
Therefore, $\rho$ cannot be faithful on $A$, which implies that $I\cap A=\ker(\rho)\cap A\neq\{0\}$.
\end{proof}

\begin{notation}
Let $(B,\sigma,D)$ be a Harnisch--Kirchberg system.\
We denote by $\I(D_\sigma)^\sigma$ the sublattice of $\I(D_\sigma)$ of $\sigma$-invariant ideals, i.e., $I\in \I(D_{\sigma})$ such that $\sigma(I) = I$.
\end{notation}

\begin{lemma}\label{lem:detectssigma}
  Let $(B,\sigma,D)$ be a Harnisch--Kirchberg system.\
  Then $D$ detects ideals in $\I(D_\sigma)^\sigma$.
\end{lemma}
\begin{proof}
We use the characterisation given in Lemma \ref{lem:representation}.\
Suppose that $\rho$ is a $\ast$-representation of $D_{\sigma}$ whose restriction to $D$ is faithful and $\ker(\rho)\in\I(D_\sigma)^{\sigma}$.\
Observe that we have the following identification,
\begin{equation*}
\sigma(\ker(\rho))=\{\sigma(x)\in D_{\sigma} \mid \rho(x)=0\}=\{x\in D_{\sigma} \mid \rho(\sigma^{-1}(x))=0\}=\ker(\rho \circ \sigma^{-1}).
\end{equation*}
Then, since $\ker(\rho)$ is $\sigma$-invariant, we have that $\ker(\rho)=\sigma^{-k}(\ker(\rho))=\ker(\rho \circ \sigma^{k})$ for all $k\in\Z$, and consequently $\rho$ is faithful on $D_{k,k}=\sigma^k(D)$ for all $k\in\Z$.\
Now, if $\ker(\rho)\cap D_{-k,k}\neq0$, then $\ker(\rho)\cap D_{-k,k}\cap D_{k,k}\neq0$ because $D_{k,k}$ is an essential ideal of $D_{-k,k}$ by Lemma \ref{lem:ideals}, which contradicts faithfulness of $\rho|_{D_{k,k}}$.\
Hence, it must be that $\ker(\rho)\cap D_{-k,k}=0$, and thus $\rho$ is faithful on $D_{\sigma}$.
\end{proof}

\begin{definition}[{cf.\ \cite[Definitions 2.8+2.15]{KM18}}]\label{def:aperiodic}
    Let $A$ be a \Cs-algebra.\
    One says that $\sigma\in\Aut(A)$ satisfies \textit{Kishimoto's condition} if for every $\e>0$, $x\in A$, and non-zero hereditary \Cs-subalgebra $H$ of $A$, there exists a positive, norm-one element $a\in H$ such that $\|ax\sigma(a)\|<\e$.
    An action $\alpha:\Gamma\curvearrowright A$ of a discrete group $\Gamma$ is said to be \textit{aperiodic} if $\alpha_g$ satisfies Kishimoto's condition for all $g\in\Gamma\setminus{\{1_{\Gamma}\}}$.
\end{definition}

\begin{remark}\label{rem:kishimoto}
When Kwa{\'s}niewski and Meyer introduced Definition \ref{def:aperiodic}, they were inspired by Kishimoto's original work on outer automorphisms \cite[Lemma 1.1]{Kis81}, thus the name ``Kishimoto's condition".
Moreover, it follows from \cite[Theorem 3.12]{GS14} (or \cite[Theorem 2.18]{KM18}) that if a discrete group action $\alpha:\Gamma\curvearrowright A$ is aperiodic, then $A$ detects ideals in the reduced crossed product $A\rtimes_{\alpha,r}\Gamma$.\
When $\Gamma=\Z$ or $\Gamma=\Z/p\Z$, where $p$ is a square-free number,\footnote{The integer $p$ is said to be square-free if no non-trivial perfect square divides $p$.} the converse is also true, and one can find a proof of this fact in \cite[Theorem 9.12]{KM18} when $A$ contains an essential ideal that either separable or type I, or in \cite[Corollary 9.6]{GU24} for the general case.\footnote{Note that in \cite{GU24}, detection of ideals is called ideal intersection property, and aperiodicity is called proper outerness.}\
Furthermore, note that if $\sigma\in\Aut(A)$, the associated action $\Z\curvearrowright A$ is aperiodic whenever for all $n>0$, the automorphism $\sigma^n$ satisfies Kishimoto's condition.\
In fact, if $\sigma^n$ with $n>0$ satisfies Kishimoto's condition, then so does $\sigma^{-n}$.
\end{remark}

\begin{lemma}\label{lem:aperiodic}
Let $(B,\sigma,D)$ be a Harnisch--Kirchberg system.\
Then the action $\Z\curvearrowright D_{\sigma}$ induced by $\sigma$ is aperiodic.\
In particular, $D_{\sigma}$ detects ideals in $D_{\sigma}\rtimes_{\sigma}\Z$.
\end{lemma}
\begin{proof}
By Remark \ref{rem:kishimoto}, it is sufficient to show that $\sigma^n$ satisfies Kishimoto's condition for all $n>0$.\
We first show that this follows from assuming that Kishimoto's condition holds for hereditary subalgebras of the form $\overline{cD_{\sigma}c}$ for $c\in D_{-i,i}$ and $i\in\N$.\
Fix $\varepsilon>0$, $n>0$, $x\in D_{\sigma}$, and a non-zero hereditary \Cs-subalgebra $H\subseteq D_\sigma$.\
Pick a positive element $b\in H$ with norm one, and a new element $b_i\in D_{-i,i}$ such that $b^2=_{\delta}b_i$ for some $i\in\N$, and $0<\delta<1$.\
It follows from \cite[Proposition 2.2+Lemma 2.3(a)]{Ror92} that there exists a positive, norm-one element $d\in D_{-i,i}$, obtained via functional calculus from $b_i$ and a nonnegative continuous function on the unit interval taking value $1$ between $2\delta$ and $1$, and $r\in D_{\sigma}$ such that $d=r^*br$.\
Picking a suitable nonnegative continuous function on the unit interval with support in $[2\delta,1]$, one can use functional calculus (applied to $b_i$) to get a positive, norm-one element $c\in D_{-i,i}$ such that $dc=c=cd$.\
Set $s=\sqrt{b}r$, and note that $s^*s=d$.\
Define a $\ast$-homomorphism $\psi:\overline{cD_{\sigma}c} \to \overline{bD_{\sigma}b}\subseteq H$ by $\psi(a)=sas^*$.\
Note that $\psi$ is isometric as $\|a\|=\|dad\|=\|s^*\psi(a)s\|\leq\|\psi(a)\|$, and set $x'=s^*x\sigma^n(s)$.\
By assumption, we can find a positive, norm-one element $a\in\overline{cD_{\sigma}c}$ such that $\|ax'\sigma^n(a)\|<\e$.\
Hence, $\|\psi(a)x\sigma^n(\psi(a))\|\leq\|as^*x\sigma^n(sa)\|=\|ax'\sigma^n(a)\|<\e$, which implies that Kishimoto's condition holds in the general situation as $\psi(a)\in\overline{bD_{\sigma}b}\subseteq H$.

It is therefore sufficient to prove that $\sigma^n$ satisfies Kishimoto's condition under the additional assumption that $H=\overline{cD_{\sigma}c}$, where $c\in D_{-i,i}$ for some $i\in\N$.\
So fix such an $H$ and $\e>0$, $n>0$, $x\in D_{\sigma}$.\
We may moreover assume that $x$ is an element of $D_{-k,k}$ for some $k\in\N$, and then by an approximation argument get the conclusion for all elements of $D_{\sigma}$.\
Note that $\overline{cD_{i,i}c}\neq\{0\}$ as $D_{i,i}$ is an essential ideal in $D_{-i,i}$ by Lemma \ref{lem:ideals}.\
Hence, we may find a non-zero, positive element $h\in\overline{cD_{i,i}c}\subseteq D_{i,i}$.\
Pick an approximate unit of positive contractions $\{e_\lambda\}_{\lambda}\subseteq D_{i,i}$ for $D_{i,i}$, and let $E_\lambda\in D_{i,i+1}$ be given by $E_\lambda=e_\lambda-\sigma(e_\lambda)$ for all $\lambda$.\
Note that by Lemma \ref{lem:approxunit} $\{e_\lambda\}_{\lambda}$ and $\{\sigma(e_\lambda)\}_{\lambda}$ are approximate units for $D_{i,\infty}$ and $D_{i+1,\infty}$, respectively, and hence it follows that for any element $y\in D_{i+1,\infty}$, the net $\{E_\lambda y\}_{\lambda}$ converges to zero.\
On the other hand, the net $\{E_\lambda h\}_{\lambda} \subseteq D_{i,i+1}$ cannot converge to zero for if $h=\lim_{\lambda}e_\lambda h=\lim_{\lambda}\sigma(e_\lambda) h \in D_{i,i}\cap D_{i+1,i+1}$, then $h=0$, a contradiction.\
Then, after passing to a subnet, we may assume that there exists $\delta>0$ such that $\|E_\lambda h\|\geq\delta$ for all $\lambda$.\
Now, $hx\sigma^n(h)$ belongs to $D_{i+n,\infty}$ by Lemmas \ref{lem:properties1} and \ref{lem:properties2}.\
Hence, we may find an index $\lambda_0$ for which $\|E_{\lambda_0}hx\sigma^n(h)\| < \frac{\e\delta^2}{2}$.\
Finally, we claim that $a\in H_+$ witnesses that Kishimoto's condition is satisfied, where
\begin{equation*} a=\frac{1}{\|E_{\lambda_0}h\|^2}(E_{\lambda_0}h)^*E_{\lambda_0}h=\frac{1}{\|E_{\lambda_0}h\|^2}hE_{\lambda_0}^2h.
\end{equation*}
Let us show the claim:
\begin{align*}
\|ax\sigma^n(a)\|&=\frac{1}{\|E_{\lambda_0}h\|^4}\|hE_{\lambda_0}^2h x \sigma^n({hE_{\lambda_0}^2h})\| \\
&\leq \frac{2}{\|E_{\lambda_0}h\|^2}\|E_{\lambda_0}h x \sigma^n(h)\| \\
&< \frac{\e\delta^2}{\|E_{\lambda_0}h\|^2}\leq \e.
\end{align*}
\end{proof}

To clarify notation in the proof of the following corollary, we should mention that the multiplier algebra of a \Cs-algebra $A$ is denoted by $\M(A)$.

\begin{corollary}[{cf.\ \cite[Proposition 4.5]{HK24}}]\label{cor:detectionD}
Let $(B,\sigma,D)$ be a Harnisch--Kirchberg system.\
Then $D$ detects ideals in $D_{\sigma}\rtimes_{\sigma}\Z$.
\end{corollary}
\begin{proof}
It follows from Lemma \ref{lem:aperiodic} that $D_{\sigma}$ detects ideals in $D_{\sigma}\rtimes_{\sigma}\Z$.\
Hence, for any non-zero ideal $I$ in $D_{\sigma}\rtimes_{\sigma}\Z$, the intersection $I\cap D_\sigma$ is non-zero.\
Moreover, $I\cap D_\sigma$ is a $\sigma$-invariant ideal of $D_{\sigma}$ because $\sigma(I\cap D_{\sigma})=\sigma(I)\cap D_{\sigma}=I\cap D_{\sigma}$, where the last identity follows from the fact that $\sigma(I)=uIu^*=I$, where $u\in \M(D_{\sigma}\rtimes_{\sigma}\Z)$ is the canonical unitary implementing $\sigma$.\
Consequently, we have that $I\cap D$ is non-zero as well by Lemma \ref{lem:detectssigma}, and thus $D$ detects ideals in $D_{\sigma}\rtimes_{\sigma}\Z$.
\end{proof}

\begin{definition}[{cf.\ \cite[Definition 4.14]{HK24}}]\label{def:semiinv}
Let $(B,\sigma,D)$ be a Harnisch--Kirchberg system.\
A closed two-sided ideal $I$ of $D$ is said to be
\begin{itemize}
\item \emph{semi-invariant} if $I\sigma(D) \subseteq \sigma(I)$,
\item \textit{cancellative} if for all $a,b\in D$ the condition $(a+\sigma(b))\sigma(D)\subseteq\sigma(I)$ implies that $a\in I$.
\end{itemize}
The set of all semi-invariant, cancellative ideals of $D$ is denoted by $\I(D)^{\rm{sic}}$.
\end{definition}

Note that the condition $I\sigma(D)\subseteq\sigma(I)$ is equivalent to saying that $I\sigma^n(D)\subseteq\sigma^n(I)$ for all $n\in\N$.
This follows from the fact that $D\sigma(D)=\sigma(D)$.

\begin{remark}\label{rem:semiinv}
Every $\sigma$-invariant ideal $I$ of $D_\sigma$, induces a semi-invariant ideal $I\cap D$ of $D$ because $(I\cap D)\sigma(D)=(\sigma(I)\cap D)\sigma(D)\subseteq\sigma(I\cap D)$.

Note that when $I\in\I(D)$ is a semi-invariant ideal, we have that $I_{k,\ell}$, $I_{-\infty,\ell}$, $I_{k,\infty}$, and $I_{\sigma}$ are \Cs-algebras for all $k, \ell\in\Z$ with $k\leq\ell$ by Lemma \ref{lem:subalgebra}.\
Moreover, they are ideals of $D_{k,\ell}$, $D_{-\infty,\ell}$, $D_{k,\infty}$ and $D_\sigma$, respectively.\
To see this, we show that $\sigma^i(I)\sigma^j(D)\subseteq\sigma^{\max(i,j)}(I)$ for all $i,j\in\Z$.
Pick $i,j\in\Z$ such that $i\geq j$,
\begin{align*}
  \sigma^i(I)\sigma^j(D) & =\sigma^j(\sigma^{i-j}(I)D) =\sigma^j(\sigma^{i-j}(ID)D) \\
  & \subseteq \sigma^j(\sigma^{i-j}(ID)) =\sigma^i(I).
\end{align*}
While if $i\leq j$ we have
\begin{equation*}
  \sigma^i(I)\sigma^j(D) =\sigma^i(I\sigma^{j-i}(D)) \subseteq \sigma^i(\sigma^{j-i}(I))=\sigma^j(I),
\end{equation*}
where we used that $I\sigma^n(D)\subseteq\sigma^n(I)$ for all $n\in\N$.
Together with the fact that $\sigma^i(D)\sigma^j(I)=(\sigma^j(I)\sigma^i(D))^*$, we have shown that $I_{k,\ell}$ is a closed two-sided ideal in $D_{k,\ell}$ for all $k\leq\ell$, and by the inductive limit decomposition the same holds for $I_{-\infty,\ell}$, $I_{k,\infty}$ and $I_\sigma$.\
Note, moreover, that $I_{\sigma}$ is clearly $\sigma$-invariant.
\end{remark}

\begin{lemma}[{cf.\ \cite[Lemma 4.15(v)]{HK24}}]\label{lem:intersection}
Let $(B,\sigma,D)$ be a Harnisch--Kirchberg system.\
Every semi-invariant ideal $J\in\I(D)$ satisfies $J_\sigma\cap D=J$.
\end{lemma}
\begin{proof}
It is evident that $J\subseteq J_\sigma\cap D$.
We want to show the opposite inclusion.\

We start by showing that $J_{-k,k} \cap D=J$ for all $k\geq 1$.\
Fix an element $x\in J_{-k,k} \cap D$, and find $y\in J_{-k,0}$, $z\in J_{1,k}$ such that $x=y+z$.\
In particular, it follows that $y=x-z$ belongs to $D_{1,k}$, but $J_{-k,0}\cap D_{1,k}=\{0\}$ by Lemma \ref{lem:independence}, and hence $y=0$ and $x=z\in J_{1,k}\cap D$.\
Now, using the same trick, we find $v\in J$ and $w\in J_{2,k}$ such that $x=v+w$, which gives that $w=(x-v) \in D$, but $D\cap J_{2,k}=\{0\}$ by Lemma \ref{lem:independence} and therefore $w=0$ and $x=v\in J$.

Consider now the quotient \Cs-algebra $D_\sigma/J_\sigma$ as the direct limit of the sequence $\{D_{-k,k}/J_{-k,k}\}_{k>1}$ with canonical inclusions.\
If $x\in J_\sigma \cap D$, then $\|x+J_\sigma\|=0$, and
\begin{equation*}
  \|x+J\|=\lim_{k\to\infty}\|x+ J_{-k,k}\cap D\|\leq\lim_{k\to\infty}\|x+ J_{-k,k}\|=\|x+J_\sigma\|=0,
\end{equation*}
which shows that $x\in J$, and finishes our proof.
%I_\sigma\cap D\subseteq I_\sigma\cap (D+\sigma(D))\cap D=(I+\sigma(I))\cap D\subseteq I
\end{proof}

\begin{notation}\label{not:newsystem}
Let $(B,\sigma,D)$ be a Harnisch--Kirchberg system, and $J$ a semi-invariant ideal of $D$.\
Then we denote by $\pi_{\scaleto{J}{4pt}}:D_{\sigma}\to D_{\sigma}/J_{\sigma}$ the quotient map, and use the notation $ B_{\scaleto{J}{4pt}} =D_{\sigma}/J_{\sigma}$, $ {D_{\scaleto{J}{4pt}}}=\pi_{\scaleto{J}{4pt}}(D)$.\
Moreover, we denote by $\sigma_{\scaleto{J}{4pt}}$ the automorphism of $ B_{\scaleto{J}{4pt}} $ given by $\sigma_{\scaleto{J}{4pt}} \circ \pi_{\scaleto{J}{4pt}} = \pi_{\scaleto{J}{4pt}} \circ \sigma$.
\end{notation}

\begin{lemma}[{cf.\ \cite[Lemma 4.15(iii)+(iv)]{HK24}}]\label{lem:newsystem}
Let $(B,\sigma,D)$ be a Harnisch--Kirchberg system, and $J\in\I(D)^{\mathrm{sic}}$.\
Then the triple $( B_{\scaleto{J}{4pt}} ,\sigma_{\scaleto{J}{4pt}}, {D_{\scaleto{J}{4pt}}})$ is a Harnisch--Kirchberg system.
\end{lemma}
\begin{proof}
Note that $ {D_{\scaleto{J}{4pt}}}\sigma_{\scaleto{J}{4pt}}( {D_{\scaleto{J}{4pt}}})=\pi_{\scaleto{J}{4pt}}(D\sigma(D))= {D_{\scaleto{J}{4pt}}}$.
Moreover, $ {D_{\scaleto{J}{4pt}}}\cap\sigma_{\scaleto{J}{4pt}}( {D_{\scaleto{J}{4pt}}})=\pi_{\scaleto{J}{4pt}}(D\cap\sigma(D))=\{0\}$ because $D\cap\sigma(D)=\{0\}$.
We want to that show ${D_{\scaleto{J}{4pt}}}\cap\sigma_{\scaleto{J}{4pt}}( {D_{\scaleto{J}{4pt}}})^{\perp}=\{0\}$.\
Hence, pick $a,b\in D$, and suppose that
\begin{equation*}
(\pi_{\scaleto{J}{4pt}}(a)+\sigma_{\scaleto{J}{4pt}}(\pi_{\scaleto{J}{4pt}}(b)))\sigma_{\scaleto{J}{4pt}}( {D_{\scaleto{J}{4pt}}})=\{0\}.
\end{equation*}
Since we have that $\pi_{\scaleto{J}{4pt}}(J_\sigma)=\{0\}$, it follows that
\begin{equation*}
(a+\sigma(b))\sigma(D) \subseteq J_\sigma\cap\sigma(D)=\sigma(J),
\end{equation*}
where we used Lemma \ref{lem:intersection} for the last identification.
By assumption $J$ is cancellative, and therefore $a\in J$, which means that $\pi_{\scaleto{J}{4pt}}(a)=0$.
As a result, we have that $\pi_{\scaleto{J}{4pt}}(b)\pi_{\scaleto{J}{4pt}}(D)=\{0\}$, which means that $\pi_{\scaleto{J}{4pt}}(b)=0$ as well, and therefore $( B_{\scaleto{J}{4pt}} ,\sigma_{\scaleto{J}{4pt}}, {D_{\scaleto{J}{4pt}}})$ is a Harnisch--Kirchberg system.
\end{proof}

\begin{proposition}[{cf.\ \cite[Proposition 4.16]{HK24}}]\label{prop:surjective}
Let $(B,\sigma,D)$ be a Harnisch--Kirchberg system.\
If $I$ is an ideal of $D_\sigma\rtimes_{\sigma}\Z$ such that $I\cap D$ is cancellative, then $I$ is the natural image of $(I\cap D)_\sigma\rtimes_{\sigma}\Z$ in $D_\sigma\rtimes_{\sigma}\Z$.
\end{proposition}
\begin{proof}
We know that $I\cap D_\sigma$ is a $\sigma$-invariant ideal of $D_\sigma$, and therefore $I\cap D$ is a semi-invariant ideal of $D$ by Remark \ref{rem:semiinv}.
Let us denote $I\cap D$ by $J$ for the rest of the proof.
By Lemma \ref{lem:newsystem}, the triple $( B_{\scaleto{J}{4pt}} ,\sigma_{\scaleto{J}{4pt}}, {D_{\scaleto{J}{4pt}}})$ associated to the ideal $J$ as in Notation \ref{not:newsystem} is a Harnisch--Kirchberg system.

The kernel of the canonical surjection
\begin{equation*}
\pi_{\scaleto{J}{4pt}}\rtimes\Z:D_\sigma\rtimes_{\sigma}\Z \to  (D_{\scaleto{J}{4pt}})_{\sigma_{\scaleto{J}{4pt}}}\rtimes_{\sigma_{\scaleto{J}{4pt}}}\Z, \quad \pi_{\scaleto{J}{4pt}}\rtimes\Z|_{D_\sigma}=\pi_{\scaleto{J}{4pt}}
\end{equation*}
is the natural image of $J_\sigma\rtimes_{\sigma}\Z$ in $D_\sigma\rtimes_{\sigma}\Z$, and it is contained in $I$ because $J_\sigma=(I\cap D)_\sigma\subseteq I$.
We want to show that the other inclusion holds as well, i.e., $I\subseteq J_\sigma\rtimes_{\sigma}\Z$.
Consider ${I}_{\scaleto{J}{4pt}}=(\pi_{\scaleto{J}{4pt}}\rtimes\Z)(I)$, which is an ideal of $ (D_{\scaleto{J}{4pt}})_{\sigma_{\scaleto{J}{4pt}}}\rtimes_{\sigma_{\scaleto{J}{4pt}}}\Z$ that satisfies ${I}_{\scaleto{J}{4pt}}\cap {D_{\scaleto{J}{4pt}}}=\pi_{\scaleto{J}{4pt}}(J)=\{0\}$.
By Corollary \ref{cor:detectionD} it follows that ${I}_{\scaleto{J}{4pt}}=\{0\}$, and thus $I\subseteq J_\sigma\rtimes_{\sigma}\Z$.
\end{proof}

The theorem that follows establishes that, under suitable assumptions, the set of semi-invariant, cancellative ideals of a Harnisch--Kirchberg system corresponds to the ideal lattice of a crossed product \Cs-algebra arising from the same Harnisch--Kirchberg system.

\begin{theorem}[{cf.\ \cite[Corollary 4.18]{HK24}}]\label{thm:orderisomorphism}
Let $(B,\sigma,D)$ be a Harnisch--Kirchberg system, and assume that every semi-invariant ideal of $D$ has cancellation.\
Then the order morphisms $\Delta:\I(D)^{\mathrm{sic}}\to D_{\sigma}\rtimes_{\sigma}\Z$ and $\nabla:D_{\sigma}\rtimes_{\sigma}\Z\to\I(D)^{\mathrm{sic}}$ given by
\begin{align*}
 \Delta(J)=J_{\sigma}\rtimes_{\sigma}\Z, \quad \nabla(I)=I\cap D
\end{align*}
for all $J\in\I(D)^{\mathrm{sic}}$ and $I\in D_{\sigma}\rtimes_{\sigma}\Z$, are mutually inverse order isomorphisms.
\end{theorem}
\begin{proof}
Fix an ideal $J\in\I^{\mathrm{sic}}(D)$.\
By Lemma \ref{lem:intersection}, we have that $J_{\sigma}\cap D=J$.\
As a consequence,
\begin{equation*}
{\nabla}\left(\Delta(J)\right) = (J_{\sigma}\rtimes_{\sigma}\Z) \cap D = J.
\end{equation*}

Now pick a non-zero ideal $I\subseteq D_{\sigma} \rtimes_{\sigma}\Z$.
By Corollary \ref{cor:detectionD}, $D$ detects ideals in $D_\sigma \rtimes_{\sigma}\Z$, and hence $I\cap D$ is a non-zero ideal of $D$.\
Moreover, $I\cap D$ is semi-invariant by Remark \ref{rem:semiinv}, and therefore also cancellative by assumption.\
From Proposition \ref{prop:surjective}, we have that $I$ coincides with the natural image of $(I\cap D)_{\sigma} \rtimes_{\sigma}\Z$ in $D_\sigma\rtimes_{\sigma}\Z$, thus showing that
\begin{equation*}
\Delta\left({\nabla}(I) \right) = (I\cap D)_{\sigma}\rtimes_{\sigma}\Z = I.
\end{equation*}
\end{proof}

\section{{Harnisch--Kirchberg system associated to a $\ast$-homomorphism}}
In this section we study the Harnisch--Kirchberg system associated to an injective, non-degenerate $\ast$-homomorphism $\f:A\to\M(A)$ such that $\f(A)\cap A=\{0\}$.\
This construction first appeared in \cite{HK24}.

\begin{definition}\label{def:sequencealgebra}
The \textit{sequence algebra} $A_{\infty}$ of a \Cs-algebra $A$ is defined as the quotient of the \Cs-algebra $\ell^{\infty}(\N,A)$ of bounded sequences with values in $A$ by the ideal of eventually vanishing sequences $c_0(\N,A)$.
\end{definition}

\begin{notation}[{cf.\ \cite[Section 4.3]{HK24}}]\label{not:HKphi}
Let $A$ be a \Cs-algebra, and $\f:A\to\M(A)$ an injective, non-degenerate $\ast$-homomorphism.\
Then $\f$ extends uniquely to an injective unital $\ast$-homomorphism $\M(A)\to\M(A)$ that is strictly continuous on the unit ball (see \cite[Propositions 2.1+2.5]{Lan95}).\
Slightly abusing notation, we denote this extension by $\f$.\
Define $\f_{\infty}:\M(A)\to\M(A)_{\infty}$ to be the $\ast$-homomorphism given by
\begin{equation*}
\f_{\infty}(a) = \left(\f^n(a)\right)_{n\in\N} + c_0\left(\N,\M(A)\right)
\end{equation*}
for all $a\in\M(A)$, where $\f^n$ denotes the $n$-fold application of $\f$.\
Moreover, we will denote by $\sigma$ the forward shift automorphism on $\M(A)_{\infty}$, i.e.,
\begin{equation*}
\sigma\left((x_n)_{n\in\N} + c_0(\N,\M(A))\right) = (x_{n-1})_{n\in\N} + c_0(\N,\M(A))
\end{equation*}
for all $(x_n)_{n\in\N} \in \ell^{\infty}\left(\N,\M(A)\right)$, with the convention that $x_n=0$ for $n<0$.
\end{notation}

\begin{lemma}[{cf.\ \cite[Lemma 4.23(i)]{HK24}}]\label{lem:phisystem}
Let $A$ be a \Cs-algebra, and $\f:A\to\M(A)$ an injective, non-degenerate $\ast$-homomorphism such that $\f(A)\cap A=\{0\}$.\
Then the triple $(\M(A)_{\infty},\sigma,\f_{\infty}(A))$ given as in Notation \ref{not:HKphi} is a Harnisch--Kirchberg system.
\end{lemma}
\begin{proof}
It follows by the definition of $\sigma$ that $\f_{\infty} \circ \f = \sigma^{-1} \circ \f_{\infty}$.\
We will use this fact multiple times below without further mentioning it.\
Since $\f$ is non-degenerate, one has that $\f(A)A=A$, and therefore
\begin{equation}\label{eq:alpha}
\f_{\infty}(A)=\f_{\infty}(\f(A)A)=\sigma^{-1}(\f_{\infty}(A))\f_{\infty}(A),
\end{equation}
which is equivalent to the first condition in Definition \ref{def:hk-system}.

The second condition is $\sigma(\f_{\infty}(A))\cap\f_{\infty}(A)=\{0\}$.\
One may obtain this relation by applying $\sigma$ to the following identity,
\begin{equation*}
\f_{\infty}(A)\cap\sigma^{-1}(\f_{\infty}(A)) = \f_{\infty}(A\cap\f(A)) = \{0\}.
\end{equation*}

Now, note that since $A$ is an essential ideal in $\M(A)$, then it is also an essential ideal of $A+\f(A)$.\
% for if $x\in (A+\f(A)) \subseteq \M(A)$ is such that $xA=0=Ax$, then $x=0$ (remember that $\M(A)$ is endowed with the strict topology).\
Then, $\f_{\infty}(A)$ is an essential ideal of $\f_{\infty}(A+\f(A))=\f_{\infty}(A)+\sigma^{-1}(\f_{\infty}(A))$, which implies that $\sigma(\f_{\infty}(A))$ is an essential ideal of $\f_{\infty}(A)+\sigma(\f_{\infty}(A))$, and therefore that $(\M(A)_{\infty},\sigma,\f_{\infty}(A))$ is a Harnisch--Kirchberg system.
\end{proof}

%% Mention that the crossed product can be seen as Cuntz-Pimsner algebra
%% CLARIFY: THIS IS NOT NECESSARY FOR RESULTS IN THIS ARTICLE.
\begin{remark}\label{rem:Cuntz-Pimsner}
Let $A$ and $\f:A\to\M(A)$ be as in Lemma \ref{lem:phisystem}.\
It was observed in \cite[Corollary 4.26(ii)]{HK24} that the Cuntz-Pimsner algebra $\mathcal{O}_{\Hil}$ \cite{Pi97} associated to $\f:A\to\M(A)=\B(\Hil)$, where $\Hil$ is the trivial right Hilbert \Cs-module over $A$, is the full hereditary \Cs-subalgebra of $(\f_{\infty}(A))_{\sigma}\rtimes_{\sigma}\Z$ generated by $\f_{\infty}(A)$.\
Therefore, if $A$ is moreover $\sigma$-unital and stable, it follows from Brown's stable isomorphism theorem \cite[Theorem 2.8]{Br77} and Proposition \ref{prop:stability} that $\mathcal{O}_{\Hil}$ and $(\f_{\infty}(A))_{\sigma}\rtimes_{\sigma}\Z$ are isomorphic.\
In the present work, we do not use this fact, and work directly in $(\f_{\infty}(A))_{\sigma}\rtimes_{\sigma}\Z$.
\end{remark}

\begin{notation}
Let $A$ be a \Cs-algebra, and $J\in\I(A)$.\
The set of all multipliers $x\in\M(A)$ such that $xA\subseteq J$ is a closed, two-sided ideal of $\M(A)$ denoted by $\M(A,J)$.\
Recall from \cite[Lemma 5.4]{KP23} that $\overline{\M(A,J)A}=\M(A,J)\cap A= J$, and that $\M(A,J)$ can be equivalently defined as the closure of $J$ in $\M(A)$ with respect to the strict topology, or as the kernel of the canonical map $\M(A)\to\M(A/J)$.
\end{notation}

\begin{notation}
Let $A$ be a \Cs-algebra, and $\f:A\to\M(A)$ a $\ast$-homomorphism.\
Denote by $\I(A)^{\f}\subseteq\I(A)$ the set consisting of all ideals $J\in\I(A)$ such that $\f(J) = \f(A)\cap\M(A,J)$.
\end{notation}

The following lemma provides necessary and sufficient conditions for an ideal of $A$ to generate a semi-invariant or cancellative ideal (see Definition \ref{def:semiinv}) in the Harnisch--Kirchberg system associated to a map $\f:A\to\M(A)$.

\begin{lemma}[{cf.\ \cite[Lemma 4.23(ii)+(iii)]{HK24}}]\label{lem:sicisomorphism}
Let $A$ be a \Cs-algebra, $\f:A\to\M(A)$ an injective, non-degenerate $\ast$-homomorphism such that $\f(A)\cap A=\{0\}$, and $(\M(A)_{\infty},\sigma,\f(A)_{\infty})$ the associated triple as in Notation \ref{not:HKphi}.\
Then, for any closed two-sided ideal $J\subseteq A$, the following statements hold true.
\begin{enumerate}[label=\textup{(\roman*)},leftmargin=*]
\item $\f(J)\subseteq \f(A)\cap\M(A,J)$ if and only if $\f(J)A\subseteq J$ if and only if $\f_{\infty}(J)$ is a semi-invariant ideal of $\f_{\infty}(A)$;
\item $\f(A)\cap(A+\M(A,J))\subseteq\f(J)$ if and only if $\f_{\infty}(J)$ is a cancellative ideal of $\f_{\infty}(A)$.
\end{enumerate}
In particular, $\f_{\infty}:\I(A)\to\I(\f_{\infty}(A))$ restricts to a bijection between $\I(A)^{\f}$ and $\I(\f_{\infty}(A))^{\mathrm{sic}}$.
\end{lemma}
\begin{proof}
Recall that $(\M(A)_{\infty},\sigma,\f(A)_{\infty})$ is a Harnisch--Kirchberg system by Lemma \ref{lem:phisystem}.
In order to prove (i), observe that the identity $\f_{\infty} \circ \f = \sigma^{-1} \circ \f_{\infty}$ implies that
\begin{align*}
\sigma^{-1}(\f_{\infty}(J))\f_{\infty}(A)=\f_{\infty}(\f(J)A).
\end{align*}
Moreover, by definition we have that
\begin{align*}
\f^{-1}(\f(A)\cap\M(A,J))=\{a\in A \mid \f(a)A\subseteq J\}.
\end{align*}
Hence, we have that $J\subseteq \f^{-1}(\f(A)\cap\M(A,J))$ if and only if $\f(J)A\subseteq J$ if and only if
\begin{align*}
\f_{\infty}(J)\sigma(\f_{\infty}(A))=\sigma(\f_{\infty}(\f(J)A)) \subseteq \sigma(\f_{\infty}(J)),
\end{align*}
which concludes part (i).

In order to prove (ii), note that for any $a\in A$, one has that $\f(a)\in\M(A,J)+A$ precisely when there exists $b\in A$ such that $(\f(a)+b)A\subseteq J$.
By applying $\sigma \circ \f_{\infty}$ to left and right hand sides, this condition becomes equivalent to
\begin{align*}
 (\f_{\infty}(a)+\sigma(\f_{\infty}(b)))\sigma(\f_{\infty}(A))\subseteq \sigma(\f_{\infty}(J)).
\end{align*}
It follows that $\f_{\infty}(J)$ is cancellative if and only if $J\supseteq \f^{-1}(\f(A)\cap(A+\M(A,J)))$.

In particular, $\f_{\infty}(J)$ is semi-invariant and cancellative precisely when
\begin{equation*}
\f(J) \subseteq \f(A)\cap\M(A,J) \subseteq \f(A)\cap(A+\M(A,J)) \subseteq \f(J),
\end{equation*}
or, equivalently, $\f(J) = \f(A)\cap\M(A,J)$.
\end{proof}

We are now able to establish, under suitable assumption on a map $\f:A\to\M(A)$, that $\I^{\f}(A)$ is naturally isomorphic to the lattice of ideals of $D_{\sigma}\rtimes_{\sigma}\Z$, where $D_{\sigma}$ comes from the Harnisch--Kirchberg system associated to $\f$.

\begin{theorem}[{cf.\ \cite[Corollary 4.18]{HK24}}]\label{thm:sicisomorphismphi}
Let $A$ be a \Cs-algebra, and $\f:A\to\M(A)$ an injective, non-degenerate $\ast$-homomorphism such that $\f(A)\cap A=\{0\}$.\
If every ideal $J\in\I(A)$ such that $\f(J)\subseteq\f(A)\cap\M(A,J)$ automatically satisfies $\f(J)=\f(A)\cap\M(A,J)$, then the order preserving map given by
\begin{equation*}
\Psi: \I(A)^{\f} \to \I((\f_{\infty}(A))_\sigma \rtimes_{\sigma} \Z), \quad J \mapsto (\f_{\infty}(J))_{\sigma}\rtimes_{\sigma}\Z
\end{equation*}
is an order isomorphism with inverse
\begin{equation*}
\Psi^{-1}: \I((\f_{\infty}(A))_\sigma \rtimes_{\sigma} \Z) \to \I(A)^{\f}, \quad I \mapsto \f_{\infty}^{-1}(I\cap \f_{\infty}(A))
\end{equation*}
\end{theorem}
\begin{proof}
To simplify notation, let us denote $\f_{\infty}(A)$ by $D$.\
From Theorem \ref{thm:orderisomorphism} we get an order isomorphism $\Delta:\I(D)^{\mathrm{sic}}\to\I(D_\sigma \rtimes_{\sigma} \Z)$ given by $\Delta(J)=J_{\sigma}\rtimes_{\sigma}\Z$ and inverse $\Delta^{-1}:\I(D_\sigma \rtimes_{\sigma} \Z)\to\I(D)^{\mathrm{sic}}$, where $\Delta^{-1}(I)=I\cap D$.\
Moreover, by Lemma \ref{lem:sicisomorphism}, we have that $\Psi = \Delta \circ \f_{\infty}$ and its inverse have the desired properties.
\end{proof}

In the last part of this section, we establish that for a separable, stable \Cs-algebra $B$ and a locally compact, second-countable group $G$, one can always find a $\ast$-homomorphism from $A=\C_0(G)\otimes B$ into its multiplier algebra giving rise to a Harnisch--Kirchberg system with certain properties that will be used in the next section.\
Let us first introduce a few preliminaries.

\begin{notation}
Here, and in the rest of the article that follows, we assume that $G$ is a locally compact, second-countable group.

Let $\mu$ be a left Haar measure on $G$.\
We denote the Hilbert spaces $L^2(G,\mu)$ and $\ell^2(\N)\hat{\otimes}L^2(G,\mu)$ by $\Hil_G$ and $\Hil_G^\infty$, respectively.\
Moreover, the left-regular representation of $G$ is denoted by $\lambda:G\to\U(\Hil_G)$, i.e., $\lambda_g(\xi)(h)=\xi(g^{-1}h)$ for all $\xi\in\Hil_G$ and $g,h\in G$, and its infinite repeat by $\lambda^\infty$, i.e., $\lambda^{\infty}=\1_{\ell^2(\N)}\otimes\lambda:G\to\U(\Hil_G^\infty)$.
\end{notation}

\begin{remark}\label{rem:compacts}
Let $B$ be a stable \Cs-algebra, and fix a sequence of isometries $s_n\in\M(B)$ such that $\sum_{n\in\N}s_ns_n^*=\1$ strictly.\
Let $\kappa:B\otimes\K \to B$ the $\ast$-isomorphism that sends $b\otimes e_{i,j}$ to $s_i b s_j^*$ for all $b\in B$, and a set of matrix units $e_{i,j}$ generating $\K$.\
One may extend this to an isomorphism $\kappa:\M(B\otimes\K)\to\M(B)$.\
In the context of this article, we use the notation $x^{\infty}=\kappa(x\otimes\1)$, and say that $x^{\infty}$ is the infinite repeat of $x$.\
Equivalently, one can define $x^{\infty}=\sum_{n\in\N}s_n x s_n^*$ with convergence in the strict topology.\
If $A\subseteq \M(B)$ is a \Cs-subalgebra, we denote by $A^{\infty}$ the set of infinite repeats of elements of $A$.\
Note that $I^{\infty}\subseteq\M(B,I)$.
Observe that the injective, non-degenerate $\ast$-homomorphism $\zeta:\K \xhookrightarrow{\1\otimes\id} \M(B\otimes\K)\xrightarrow{\kappa}\M(B)$ is given on matrix units by $\zeta(e_{i,j})=s_is_j^*$, and its image lies in $\M(B)\cap(\M(B)^{\infty})'$.\
It follows that $\zeta$ extends uniquely to an injective unital $\ast$-homomorphism $\M(\K)\to\M(B)\cap(\M(B)^{\infty})'$ that is strictly continuous on the unit ball (see, e.g., \cite[Propositions 2.1+2.5]{Lan95}).
%Moreover, it follows from \cite[Proposition 1.7]{BS16} that the extension of $\zeta$ to $\M(\K)$ is an isomorphism onto $\M(B)\cap(\M(B)^{\infty})'$.
\end{remark}

\begin{theorem}[{cf.\ \cite[Theorem 5.7]{KP23}}]\label{thm:existencephi}
Let $B$ be a separable, stable \Cs-algebra, and denote the \Cs-algebra $\C_0(G)\otimes B$ by $A$.
There exists a non-degenerate, injective $\ast$-homomorphism $\f:A \to \M(A)$ with $\f(A)\cap A=\{0\}$, and such that the following are equivalent for any $J\in\I(A)$,
\begin{enumerate}[label=\textup{(\roman*)},leftmargin=*]
\item $\f(J)\subseteq\f(A)\cap\M(A,J)$,
\item $J=\C_0(G)\otimes I$ for some $I\in\I(B)$,
\item $\f(J)=\f(A)\cap\M(A,J)$.
\end{enumerate}
\end{theorem}
\begin{proof}
In this proof, we denote by $\K$ the compact operators on $\Hil_G^{\infty}=\ell^2(\N)\hat{\otimes}L^2(G)$.\
Define a non-degenerate, injective $\ast$-homomorphism given by
\begin{equation*}
\psi:\C_0(G) \xhookrightarrow{\pi} \M(\K) \xrightarrow{\zeta} \M(B) \cap (\M(B)^{\infty})',
\end{equation*}
where $\zeta$ is an isomorphism as in Remark \ref{rem:compacts}, and $\pi:\C_0(G)\to\M(\K)$ is the $\ast$-representation that lets elements of $\C_0(G)$ act as multiplication operators on $L^2(G)$ and trivially on $\ell^2(\N)$.
One can therefore define the following non-degenerate $\ast$-homomorphism
\begin{equation*}
\f: \C_0(G)\otimes B \to \M(\C_0(G)\otimes B), \quad f\otimes b \mapsto \1 \otimes \psi(f) \cdot b^{\infty}.
\end{equation*}
To see that $\f$ is faithful, it suffices to observe that the image of $\psi$ under the isomorphism $\M(B\otimes\K)\cong\M(B)$ in Remark \ref{rem:compacts} lies in $\M(\K)\otimes\1_{\M(B)}$, while that of $\M(B)^{\infty}$ lies in $\1_{\M(\K)}\otimes\M(B)$.\
Moreover, $\f(A)\cap A=\{0\}$ because $\f(A)\subseteq\1\otimes\M(B)$ and $(\1\otimes\M(B))\cap A=\{0\}$.\

Since (iii)$\Rightarrow$(i) is always true, we only need to prove that (i)$\Rightarrow$(ii)$\Rightarrow$(iii).

We first show that for any ideal $I_1\in\I(\C_0(G))$ and $I_2\in\I(B)$, one has that $\overline{A\f(I_1\otimes I_2)A}=\C_0(G)\otimes I_2$.\
This follows from the fact that $A\f(x_1\otimes x_2)A=\C_0(G)\otimes(B\psi(x_1)x_2^{\infty}B)$, and that $\overline{B\psi(x_1)x_2^{\infty}B}=\overline{Bx_2B}$ for all $x_1\in I_1$ and $x_2\in I_2$.\
The second observation holds true because the image of the ideal $\overline{B\psi(x_1)x_2^{\infty}B}$ under the isomorphism $\M(B\otimes\K)\cong\M(B)$ of Remark \ref{rem:compacts} is $\overline{Bx_2B}\otimes\K$.

We want to show that (i)$\Rightarrow$(ii).\
Pick an ideal $J\in\I(A)$ satisfying $\f(J)\subseteq\f(A)\cap\M(A,J)$.\
For each group element $g\in G$, let $I_g$ be the ideal of $B$ given by $\I(\ev_g)(J)$.\
Note that $J\subseteq \C_0(G)\otimes\overline{\sum_{g\in G}I_g}$, and we want to prove that the opposite inclusion holds as well.\
Since $B$ is separable and $\C_0(G)$ exact, for each $g\in G$, one may find a full element $b$ of $I_g\in\I(B)$, and a non-zero function $f\in\C_0(G)$ such that $f\otimes b \in J$ (see, e.g., \cite[Corollary 9.4.6]{BO08}).
One may then conclude that $\C_0(G)\otimes I_g = \overline{A\f(f\otimes b)A} \subseteq J$, where the last inclusion follows from the fact that $J\subseteq \f^{-1}(\f(A)\cap\M(A,J))$.\
As a consequence, one gets that $J = \C_0(G)\otimes I$, where $I = \overline{\sum_{g\in G}I_g}$.

In order to show that (ii)$\Rightarrow$(iii), pick an ideal of the form $J = \C_0(G)\otimes I$ for some $I\in\I(B)$.\
First, we observe that
\begin{equation*}
  \f(J)A = \C_0(G)\otimes(\psi(\C_0(G))\cdot I^{\infty}\cdot B) \subseteq J,
\end{equation*}
and hence $J\subseteq \f^{-1}(\f(A)\cap\M(A,J))$.\
Pick now an element $(f\otimes b)\in \f^{-1}(\f(A)\cap\M(A,J))$ for some $f\in\C_0(G)$ and $b\in B$.\
The condition $\f(f\otimes b)A\subseteq J$ implies that $\psi(f)b^{\infty}B\subseteq I$, and by the claim proved before, that $\overline{BbB}\subseteq I$.\
In particular, $b\in I$.\
Consequently, $(f\otimes b)\in J$, and therefore $J = \f^{-1}(\f(A)\cap\M(A,J))$.\
This implies that $J$ satisfies $\f(J)=\f(A)\cap\M(A,J)$.
\end{proof}

\section{Main theorem and an application to classification}
In the present section, we prove that every continuous action on the primitive ideal space of a separable, nuclear \Cs-algebra lifts to a \Cs-action as described in the introduction.\
By combining this theorem with classification results in \cite{PS23}, we get an explicit bijection between continuous actions on primitive ideal spaces, identified up to conjugacy, and amenable, isometrically shift-absorbing, equivariantly $\O2$-stable actions up to cocycle conjugacy (or conjugacy when $G$ is compact).

\begin{lemma}\label{lem:ideals-preimage}
Let $A$ be a \Cs-algebra, and $\f_1,\f_2:A\to\M(A)$ two $\ast$-homomorphisms.\
If $\overline{A\f_1(J)A}=\overline{A\f_2(J)A}$ for an ideal $J\in\I(A)$, then $\f_1^{-1}(\f_1(A)\cap\M(A,J))=\f_2^{-1}(\f_2(A)\cap\M(A,J))$.
\end{lemma}
\begin{proof}
By assumption, $\overline{A\f_1(J)A}=\overline{A\f_2(J)A}$ for an ideal $J\in\I(A)$.\
Pick an element $a\in\f_1^{-1}(\f_1(A)\cap\M(A,J))$.\
We have that $\f_2(a)A \subseteq \overline{A\f_2(a)A} = \overline{A\f_1(a)A} \subseteq J,$ which means that $a\in \f_2^{-1}(\f_2(A)\cap\M(A,J))$.\
By switching the role of $\f_1$ and $\f_2$, one shows that if $a\in \f_2^{-1}(\f_2(A)\cap\M(A,J))$ then it must be $a\in\f_1^{-1}(\f_1(A)\cap\M(A,J))$.\
In particular, $\f_1^{-1}(\f_1(A)\cap\M(A,J))=\f_2^{-1}(\f_2(A)\cap\M(A,J))$.
\end{proof}

\begin{lemma}\label{lem:tri-implication}
Let $B$ be a \Cs-algebra, $A=\C_0(G)\otimes B$, and $\f_1,\f_2:A\to\M(A)$ two $\ast$-homomorphisms such that  $\f_1^{-1}(\f_1(A)\cap\M(A,J))=\f_2^{-1}(\f_2(A)\cap\M(A,J))$ for all $J\in\I(A)$.\
Consider the following properties for an ideal $J\in\I(A)$, and $i=1,2$
\begin{enumerate}[label=\textup{(\roman*)},leftmargin=*]
\item $\f_i(J)\subseteq\f_i(A)\cap\M(A,J)$,\label{i1}
\item $J=\C_0(G)\otimes I$ for some $I\in\I(B)$,\label{i2}
\item $\f_i(J)=\f_i(A)\cap\M(A,J)$.\label{i3}
\end{enumerate}
If \ref{i1}$\Leftrightarrow$\ref{i2}$\Leftrightarrow$\ref{i3} for $i=1$, the same holds for $i=2$.
\end{lemma}
\begin{proof}
Assume that \ref{i1}$\Leftrightarrow$\ref{i2}$\Leftrightarrow$\ref{i3} holds true for $i=1$, and suppose that \ref{i1} holds for $i=2$, i.e., $\f_2(J)\subseteq \f_2(A)\cap\M(A,J)$ for some ideal $J\in\I(A)$.\
One can re-write this condition as $J\subseteq \f_2^{-1}(\f_2(A)\cap\M(A,J))=\f_1^{-1}(\f_1(A)\cap\M(A,J))$, and thus $\f_1(J)\subseteq\f_1(A)\cap\M(A,J)$.\
Since, by assumption, this is equivalent to  $J=\C_0(G)\otimes I$ for some $I\in\I(B)$, and to $J=\f_1^{-1}(\f_1(A)\cap\M(A,J))=\f_2^{-1}(\f_2(A)\cap\M(A,J))$, we have that \ref{i1}$\Leftrightarrow$\ref{i2}$\Leftrightarrow$\ref{i3} holds also when $i=2$.
\end{proof}

\begin{remark}\label{rem:idealproduct}
Let $A$ and $B$ be \Cs-algebras.\
It is shown in \cite[Propositions 2.16(iii)+2.17(2)]{BK04} that an homeomorphism $\Prim(A)\times\Prim(B)\cong\Prim(A\otimes B)$ when either $A$ or $B$ is exact is given by
\begin{equation*}
\Prim(A)\times\Prim(B) \to \Prim(A\otimes B), \quad (I,J) \mapsto (I\otimes B) + (A\otimes J)
\end{equation*}
for all $I\in\Prim(A)$ and $J\in\Prim(B)$.
\end{remark}

The following result is a generalisation of Theorem \ref{thm:existencephi} to the dynamical setting.\
In particular, it provides sufficient conditions for the resulting map $\f:A\to\M(A)$ to be equivariant with respect to a given action.

\begin{theorem}[{cf.\ \cite[Theorem 5.8]{KP23}}]\label{thm:existencephi2}
Let $B$ be a separable, stable \Cs-algebra, and denote the \Cs-algebra $\C_0(G)\otimes B$ by $A$.\
Let $\alpha:G\curvearrowright A$ be an action that is conjugate to $\Ad(\lambda^{\infty})\otimes\alpha:G\curvearrowright\K(\Hil_G^{\infty})\otimes A$, and such that $\alpha^{\sharp}=\tau\times\gamma:G\curvearrowright G\times\Prim(B)$, where $\tau:G\curvearrowright G$ and $\gamma:G\curvearrowright\Prim(B)$ are continuous actions.\footnote{By Remark \ref{rem:idealproduct}, $\Prim(A)\cong G\times\Prim(B)$.\ The action $\tau\times\gamma$ is the continuous action given by acting componentwise on the product $G\times\Prim(B)$.}
Then, there exists an $\alpha$-to-$\alpha$-equivariant, non-degenerate, injective $\ast$-homomorphism $\f:A \to \M(A)$ with $\f(A)\cap A=\{0\}$, and such that the following are equivalent for any ideal $J\in\I(A)$,
\begin{enumerate}[label=\textup{(\roman*)},leftmargin=*]
\item $\f(J)\subseteq\f(A)\cap\M(A,J)$,
\item $J=\C_0(G)\otimes I$ for some $I\in\I(B)$,
\item $\f(J)=\f(A)\cap\M(A,J)$.
\end{enumerate}
\end{theorem}
\begin{proof}
By Theorem \ref{thm:existencephi}, there exists an injective, non-degenerate $\ast$-homomorphism $\f_0:A\to\M(A)$ such that $\f_0(A)\cap A=\{0\}$, and such that conditions (i), (ii), (iii) above are equivalent for any ideal $J\in\I(A)$.\
Consider the injective, non-degenerate $\ast$-homomorphism
\begin{equation*}
 \hat{\f}_0:A\to\C_b^{s}(G,\M(A))=\M(\C_0(G)\otimes A),\quad \hat{\f}_0(a)(g) = \alpha_{g}(\f_0(\alpha_{g^{-1}}(a)))
\end{equation*}
for all $g\in G$, and $a\in A$, where $\C_b^{s}(G,\M(A))$ denotes the \Cs-algebra of bounded, strictly continuous functions from $G$ to $\M(A)$.\
Moreover, we have by definition that $\hat{\f}_0(A)\cap(\C_0(G)\otimes A)=\{0\}$.\
If we equip $\C_b^{s}(G,\M(A))$ with the $G$-action $\bar{\alpha}$ given by
\begin{equation*}
\bar{\alpha}_g(f)(h) = \alpha_g(f(g^{-1}h))
\end{equation*}
for all $g,h\in G$, and $f\in \C_b^{s}(G,\M(A))$, it follows that $\hat{\f}_0$ is $\alpha$-to-$\bar{\alpha}$-equivariant.\
By Remark \ref{rem:idealproduct}, any ideal $J\in\Prim(A)$ can be written as $J = \C_0(G)\otimes I+\C_0(G\setminus h)\otimes B$ for some $I\in\Prim(B)$ and $h\in G$.\
By definition of $\f_0$ (see proof of Theorem \ref{thm:existencephi}) we have that $\overline{A\f_0(J)A} = \C_0(G)\otimes I$, and therefore that
\begin{equation*}
\overline{A\f_0(\alpha^{\sharp}_g(J))A} = \C_0(G)\otimes\gamma_g(I) = \alpha^{\sharp}_g\left(\overline{A\f_0(J)A}\right)
\end{equation*}
for all $g\in G$.\
Pick now an ideal $J\in\I(A)$.\
We know that $J$ can be written as $J=\bigcap_{\lambda\in\Lambda}J_\lambda$, where $J_\lambda\in\Prim(A)$ for all $\lambda\in\Lambda$.\
Then,
\begin{align*}
\alpha^{\sharp}_g(\overline{A\f_0(J)A})
=\bigcap_{\lambda}\alpha^{\sharp}_g\left(\overline{A\f_0(J_\lambda)A}\right)
=\bigcap_{\lambda}\overline{A\f_0(\alpha^{\sharp}_g(J_\lambda))A}
=\overline{A\f_0(\alpha^{\sharp}_g(J))A}
\end{align*}
for all $g\in G$.\
As a consequence, $\overline{A\f_0(-)A}:\I(A)\to\I(A)$ is an $\alpha^{\sharp}$-equivariant Cuntz morphism (in the sense of \cite{PS23}), and we may apply \cite[Lemma 4.6]{PS23} to conclude that
\begin{equation*}
\overline{(\C_0(G)\otimes A)\hat{\f}_0(J)(\C_0(G)\otimes A)} = \C_0(G)\otimes\overline{A\f_0(J)A}
\end{equation*}
for all $J\in\I(A)$.\

By assumption, there exists a conjugacy $\eta:(\M(\K\otimes A),\Ad(\lambda^{\infty})\otimes\alpha)\to (\M(A),\alpha)$, where $\K$ is tacitly assumed to be $\K(\Hil_G^{\infty})$.\
Consider the $\ast$-homomorphism $\f:A\to\M(A)$ given by the composition of the following maps,
\begin{equation*}
A \xhookrightarrow{\hat{\f}_0} \C_b^{s}(G,\M(A)) = \M(\C_0(G,A)) \xhookrightarrow{\pi\otimes\id_{A}} \M(\K\otimes A) \xrightarrow{\eta} \M(A),
\end{equation*}
where $\pi:\C_0(G)\to\M(\K)$ is the $\ast$-representation as multiplication operators, and $\pi\otimes\id_{A}$ is $\bar{\alpha}$-$(\Ad(\lambda^{\infty})\otimes\alpha)$ equivariant.\
Note that $\f$ is an injective, non-degenerate, $\alpha$-to-$\alpha$ equivariant $\ast$-homomorphism because it is a composition of injective, non-degenerate, equivariant $\ast$-homomorphisms.\
Moreover, one can compute that $\f(A) \cap A = \{0\}$ holds true ultimately because $\f_0(A) \cap A = \{0\}$.

Now, we want to prove that (i)$\Leftrightarrow$(ii)$\Leftrightarrow$(iii) holds true for the map $\f$.\
By combining Lemma \ref{lem:ideals-preimage} and Lemma \ref{lem:tri-implication}, it is sufficient to show that the $\overline{A\f(J)A}=\overline{A\f_0(J)A}$ for every ideal $J\in\I(A)$.\
Fix an ideal $J\in\I(A)$.\
First notice that the ideal generated by $\eta \left( \pi(\C_0(G))\otimes A\right)$ in $A$ is $A$.\
This implies that the ideal generated by $\f(J)$ in $A$ coincides with the ideal generated by the image of $\overline{(\C_0(G)\otimes A)\hat{\f}_0(J)(\C_0(G)\otimes A)}$ under $\eta \circ (\pi\otimes\id_A)$.\
From the observation above we have that $\overline{(\C_0(G)\otimes A)\hat{\f}_0(J)(\C_0(G)\otimes A)}$ is $\C_0(G)\otimes\overline{A\f_0(J)A}$, and therefore $\overline{A\f(J)A}=\overline{A\f_0(J)A}$.
\end{proof}

\begin{notation}
Let $A$ be a \Cs-algebra equipped with an action $\alpha:G\curvearrowright A$, and an automorphism $\sigma$ such that $\alpha_g \circ \sigma=\sigma \circ \alpha_g$ for all $g\in G$.\
We denote by $\alpha\rtimes\Z$ the $G$-action on $A\rtimes_{\sigma}\Z$ that extends $\alpha$ and acts trivially on the canonical unitary representing $\Z$.
\end{notation}

\begin{remark}\label{rem:continuity induced action}
Let $A$ be a \Cs-algebra equipped with an action $\alpha:G\curvearrowright A$.\
The action $\alpha$ induces an algebraic action $\alpha^{\sharp}:G\curvearrowright\I(A)$ given by $\alpha^{\sharp}_g(I)=\{\alpha_g(x) \mid x\in I\}$, or equivalently by $\alpha^{\sharp}_g(I)=\ker(\pi \circ \alpha_{g^{-1}})$, where $I=\ker(\pi)$ for some non-zero irreducible representation $\pi$ of $A$.\
It is a well-known fact that $\alpha^{\sharp}$ restricts to a continuous action on $\Prim(A)$, which we call again by $\alpha^{\sharp}$, with slight abuse of notation.\
A proof of this fact can be found in \cite[Lemma 7.1]{RW98}, and a different proof is carried out in \cite[Lemma 2.25]{PS23}.
\end{remark}

Now we can prove Theorem \ref{thm-intro:realisation}, the main result of this article.

\begin{proof}[Proof of Theorem \ref{thm-intro:realisation}]
Denote the \Cs-algebra $\C_0(G)\otimes B$ by $A$, and let $\K$ be concretely represented as $\K(\Hil_G^{\infty})$.\
Since $\gamma$ is continuous between $G\times\Prim(B)$ and $\Prim(B)$, the map given by
\begin{equation*}
G\times\Prim(B) \to G\times\Prim(B), \quad (g,\p) \mapsto (g,\gamma_g(\p))
\end{equation*}
for all $g\in G$, and $\p\in\Prim(B)$, is a homeomorphism.\
By composing this map with the canonical homeomorphism between $G\times\Prim(B)$ and $\Prim(A)$, we obtain a homeomorphism $\Theta:\Prim(A)\to\Prim(A)$.\
By \cite[Theorem 6.13]{Gab20} $\Theta$ is induced by an automorphism $\theta\in\Aut(A)$.

Since $B$ is stable by assumption, there exist an isomorphism $\kappa:\M(B\otimes\K)\hookrightarrow\M(B)$ and an embedding $\zeta:\M(\K)\to\M(B)$ as in Remark \ref{rem:compacts} such that $\kappa(\1\otimes\lambda^{\infty}_g)=\zeta(\lambda^{\infty}_g)$.\
Denote by $\rt$ the $G$-action on $\C_0(G)$ induced by the right translation on $G$, i.e., $\rt_g(f)(h)=f(hg)$ for all $f\in\C_0(G)$, $g,h\in G$.\
It is now possible to define an action $\alpha:G\curvearrowright A$ as follows,
\begin{equation*}
\alpha_g = \theta^{-1} \circ \left(\rt_g \otimes \Ad(\zeta(\lambda^{\infty}_g))\right) \circ \theta
\end{equation*}
for all $g\in G$.
It follows that $\alpha$ is conjugate to $\Ad(\lambda^{\infty})\otimes\alpha:G\curvearrowright\K\otimes A$ by construction.

Note that the action induced by $\alpha$ on $\Prim(A)$ is given by
\begin{equation*}
\alpha^{\sharp}_g(\C_0(G)\otimes I + \C_0(G\setminus h)\otimes B) = \C_0(G)\otimes\gamma_g(I) + \C_0(G\setminus{hg^{-1}})\otimes B,
\end{equation*}
for all $g,h\in G$ and $I\in\Prim(B)$, where we used Remark \ref{rem:idealproduct}.\
Therefore, $\alpha^{\sharp}$ is $\rt^{\sharp}\otimes\gamma$ on $\Prim(A)$.\
We may apply Theorem \ref{thm:existencephi2} to obtain an $\alpha$-equivariant, injective, non-degenerate $\ast$-homomorphism $\f:A\to\M(A)$ such that $\f(A)\cap A=\{0\}$ and with the following property.\
Any ideal $J\in\I(A)$ such that $\f(J)\subseteq\f(A)\cap\M(A,J)$ already satisfies $\f(J)=\f(A)\cap\M(A,J)$ and is of the form $J=\C_0(G)\otimes I$ for some $I\in\I(B)$.\
In particular, we have a natural order isomorphism $\Phi:\I(B)\to\I(A)^{\f}$, given by $\Phi(I)=\C_0(G)\otimes I$ for all $I\in\I(B)$.

By Theorem \ref{thm:sicisomorphismphi}, there exists a natural order isomorphism $\Psi:\I(A)^{\f}\to \I((\f_{\infty}(A))_{\sigma}\rtimes_{\sigma}\Z)$ given by $\Psi(J)=(\f_{\infty}(J))_{\sigma}\rtimes_{\sigma}\Z$ for all $J\in\I(A)^{\f}$, with inverse $\Psi^{-1}(I)=\f_{\infty}^{-1}(I\cap\f_{\infty}(A))$ for all $I\in\I(\f_{\infty}(A)_{\sigma}\rtimes_{\sigma}\Z)$.\
Hence, we have an order isomorphism $\Psi \circ \Phi : \I(B) \to \I((\f_{\infty}(A))_{\sigma}\rtimes_{\sigma}\Z)$.\
Note that the crossed product $(\f_{\infty}(A))_{\sigma}\rtimes_{\sigma}\Z$ is separable from separability of $A$, nuclear because arises from an amenable group action on a nuclear \Cs-algebra, and stable by Proposition \ref{prop:stability}.\
Hence, one can apply the classification theorem for $\O2$-stable \Cs-algebras of Gabe and Kirchberg (see \cite{Kir} or \cite[Theorem 6.13]{Gab20}) to obtain a $\ast$-isomorphism $\eta: B \to \left((\f_{\infty}(A))_{\sigma}\rtimes_{\sigma}\Z\right) \otimes\O2$ inducing $\I(\id\otimes\1_{\O2}) \circ \Psi \circ \Phi$.\footnote{Here, $\I(\id\otimes\1_{\O2})$ denotes the order isomorphism $\I(((\f_{\infty}(A))_{\sigma}\rtimes_{\sigma}\Z))\to\I(((\f_{\infty}(A))_{\sigma}\rtimes_{\sigma}\Z)\otimes\O2)$ that sends an ideal $I$ to $I\otimes\O2$.}

Denote by $\beta:G\curvearrowright B$ the action given by $\beta_g = \eta^{-1} \circ ((\alpha_{\infty}\rtimes\Z)\otimes\id_{\O2})_g \circ \eta$ for all $g\in G$.\
Then $\beta$ satisfies the following identity,
\begin{align*}
\beta^{\sharp}_g(I) &= \Phi^{-1} \circ \Psi^{-1} \circ (\alpha\rtimes\Z)^{\sharp}_g \left(\left(\f_{\infty}(\C_0(G)\otimes I)\right)_{\sigma}\rtimes_{\sigma}\Z\right) \\
&= \Phi^{-1} \circ \Psi^{-1} \left(\left(\f_{\infty}(\C_0(G)\otimes \gamma_g(I))\right)_{\sigma}\rtimes_{\sigma}\Z\right) \\
&= \Phi^{-1} (\C_0(G)\otimes \gamma_g(I)) = \gamma_g(I)
\end{align*}
for all $g\in G$, and $I\in\I(B)$.\
It follows that $\beta^{\sharp}=\gamma$ as actions on $\Prim(B)$, which concludes the proof.
\end{proof}

The following definition makes sense for any separable, unital, strongly self-absorbing \Cs-algebra $\mathcal{D}$ (see \cite[Definition 1.3]{TW07}), but we only need the case $\mathcal{D}=\O2$.

\begin{definition}[{see \cite[Definition 3.1]{Sza18a} and \cite[Definition 5.3]{Sza21}}]
An action $\alpha:G\curvearrowright A$ on a separable \Cs-algebra is said to be \emph{equivariantly $\O2$-stable} if $\alpha\otimes\id_{\O2}$ is cocycle conjugate to $\alpha$, where $\id_{\O2}$ denotes the trivial $G$-action on $\O2$.
\end{definition}

Isometric shift-absorbing actions on separable \Cs-algebras are introduced in \cite[Definition C]{GS22b}, and we refer the reader to the same reference for more information about these actions.
We also do not define amenability or the quasicentral approximation property \cite{Suz19}, which are equivalent by \cite{BEW20,BEW20a,OS21}.\
Suffice it to say that all dynamical properties mentioned above pass to tensor products.

\begin{remark}\label{rem:amenable action on O2}
We give a brief argument explaining that any second-countable, locally compact group admits an amenable, isometrically shift-absorbing, equivariantly $\O2$-stable action on $\O2\otimes\K$.\
First, recall that by \cite[Theorem 6.1]{OS21}, for any amenable action $\alpha:G\curvearrowright A$ on a separable, nuclear \Cs-algebra, there exists an action $\beta:G\curvearrowright B$ on a purely infinite simple, separable, nuclear \Cs-algebra containing $(A,\alpha)$ as a subsystem.\
In particular, since any locally compact, second-countable group $G$ acts amenably on $A=\C_0(G)$, one may find an amenable action $\beta:G\curvearrowright B$ as above.\
Moreover, we may pick any isometrically shift-absorbing action $\delta:G\curvearrowright\Oinf$ such as  $\delta=\gamma^{\otimes\infty}$ from \cite[Definition 3.4]{GS22b}.\
Since amenability and isometric shift-absorption are preserved under tensoring with an action, then $\beta\otimes\delta\otimes\id_{\O2\otimes\K}$ is an amenable, isometrically shift-absorbing and equivariantly $\O2$-stable action on $B\otimes\Oinf\otimes\O2\otimes\K\cong\O2\otimes\K$, where the last isomorphism follows from Kirchberg's $\O2$-absorption theorem \cite{Kir,KP00} and Zhang's dichotomy \cite{Zha92}.

If the acting group is moreover exact, one can always find an amenable, isometrically shift-absorbing, equivariantly $\O2$-stable action on $\O2$.\
One way of seeing this is by picking an amenable $G$-action on $\Oinf$, which exists by \cite[Theorem B]{Suz24}, and then tensoring this action with $\delta\otimes\id_{\O2}$, where $\delta:G\curvearrowright\Oinf$ is as in the previous paragraph.\
We seize the opportunity to rectify an incorrect reference in \cite[Remarks 4.22+5.9]{PS23}, where the same statement as above is remarked with an (erroneous) reference to \cite[Theorem 6.6]{OS21} where it should be \cite[Theorem B]{Suz24}.
\end{remark}

We can now prove Theorem \ref{cor-intro} as an application of Theorem \ref{thm-intro:realisation} combined with the classification theorem for equivariantly $\O2$-stable actions in \cite{PS23}.

\begin{proof}[Proof of Theorem \ref{cor-intro}]
Let $A$ be separable, nuclear, stable and $\O2$-stable.\
By \cite[Theorem 5.6]{PS23}, an amenable, isometrically shift-absorbing, equivariantly $\O2$-stable action $\alpha:G\curvearrowright A$ is classified, up to cocycle conjugacy, by the induced action $\alpha^{\sharp}:G\curvearrowright\Prim(A)$.\
Hence, by associating $\alpha^{\sharp}$ to $\alpha$, one gets that the map
\begin{equation*}
\frac{\begin{Bmatrix}\text{\rm{amenable, isometrically shift-absorbing,}}\\ \text{\rm{equivariantly $\O2$-stable actions }}G\curvearrowright A \end{Bmatrix}}{\text{\rm{cocycle conjugacy}}} \longrightarrow \frac{\begin{Bmatrix}\text{\rm{continuous actions}}\\ G\curvearrowright\Prim(A) \end{Bmatrix}}{\text{\rm{conjugacy}}}
\end{equation*}
is injective.\
To see that it is also surjective, note that by Theorem \ref{thm-intro:realisation}, given a continuos action $\sigma:G\curvearrowright\Prim(A)$, there exists an action $\alpha:G\curvearrowright A$ such that $\alpha^{\sharp}=\sigma$.\
Arguing as in Remark \ref{rem:amenable action on O2}, one can find an amenable, isometrically shift-absorbing, equivariantly $\O2$-stable action $\beta:G\curvearrowright \O2\otimes\K$.\
Then $\alpha':=\alpha\otimes\beta$ is an amenable, isometrically shift-absorbing, equivariantly $\O2$-stable action on $A\otimes\O2\otimes\K\cong A$ such that $(\alpha')^{\sharp}$ is conjugate to $\sigma$.\
This concludes the first part of the proof.

If $G$ is compact, the bijection above holds true when every instance of ``cocycle conjugacy" is replaced by ``conjugacy", in which case one needs to apply \cite[Corollary 5.8]{PS23} instead of \cite[Theorem 5.6]{PS23} in the previous paragraph.
\end{proof}

%%%%%

\textbf{Acknowledgements.}
The author was supported by PhD-grant 1131623N funded by the Research Foundation Flanders (FWO).\
He is grateful to G.\ Szabó for valuable discussions on the topic of this article, and to N.\ C.\ Phillips for kindly sharing a copy of his work with E.\ Kirchberg.\
Moreover, he would like to thank the anonymous referee for their remarks that led to numerous little improvements in this article.

\bibliography{references.bib}

\begin{thebibliography}{10}
\providecommand{\url}[1]{\texttt{#1}}
\providecommand{\urlprefix}{URL }
\expandafter\ifx\csname urlstyle\endcsname\relax
  \providecommand{\doi}[1]{doi:\discretionary{}{}{}#1}\else
  \providecommand{\doi}{doi:\discretionary{}{}{}\begingroup
  \urlstyle{rm}\Url}\fi

\bibitem{BK04}
{\'E}.~Blanchard and E.~Kirchberg.
\newblock Non-simple purely infinite $\Cs$-algebras: the Hausdorff case.
\newblock J. Funct. Anal. 207 (2004), no.~2, pp. 461--513.

\bibitem{Br77}
L.~Brown.
\newblock Stable isomorphism of hereditary subalgebras of \Cs-algebras.
\newblock Pac. J. Math. 71 (1977), no.~2, pp. 335--348.

\bibitem{BO08}
N.~P. Brown and N.~Ozawa.
\newblock {$\Cs$}-algebras and finite-dimensional approximations,
  \emph{Graduate Studies in Mathematics}, vol.~88.
\newblock Amer. Math. Soc. (2008).

\bibitem{BEW20a}
A.~Buss, S.~Echterhoff and R.~Willett.
\newblock Injectivity, crossed products, and amenable group actions.
\newblock Contemp. Math. 749 (2020), pp. 105--137.

\bibitem{BEW20}
A.~Buss, S.~Echterhoff and R.~Willett.
\newblock Amenability and weak containment for actions of locally compact
  groups on \Cs-algebras.
\newblock Mem. Amer. Math. Soc. 301 (2024).
\newblock 88 pp.

\bibitem{Gab20}
J.~Gabe.
\newblock A new proof of Kirchberg's $\mathcal{O}_2$-stable classification.
\newblock J. reine angew. Math. 761 (2020), pp. 247--289.

\bibitem{GS22b}
J.~Gabe and G.~Szab{\'o}.
\newblock The dynamical Kirchberg--Phillips theorem.
\newblock Acta Math. 232 (2024), no.~1, pp. 1--77.

\bibitem{GS22a}
J.~Gabe and G.~Szab{\'o}.
\newblock The stable uniqueness theorem for equivariant Kasparov theory.
\newblock To appear in Amer. J. Math.  (2024).
\newblock \urlprefix\url{https://arxiv.org/abs/2202.09809}, 43 pp.

\bibitem{GU24}
S.~Geffen and D.~Ursu.
\newblock Simplicity of crossed products by FC-hypercentral groups  (2024).
\newblock \urlprefix\url{https://arxiv.org/abs/2304.07852}, preprint, 63 pp.

\bibitem{GS14}
T.~Giordano and A.~Sierakowski.
\newblock Purely infinite partial crossed products.
\newblock J. Funct. Anal. 266 (2014), no.~9, pp. 5733--5764.

\bibitem{HK24}
H.~Harnisch and E.~Kirchberg.
\newblock The inverse problem for primitive ideal spaces  (2024).
\newblock \urlprefix\url{https://arxiv.org/abs/2401.05917}, preprint, 67 pp.

\bibitem{Kir}
E.~Kirchberg.
\newblock The Classification of Purely Infinite \Cs-Algebras Using Kasparov's
  Theory.
\newblock
  \urlprefix\url{https://www.uni-muenster.de/MathematicsMuenster/events/2023/cstar-algebras.shtml},
  preprint, 1348 pp.

\bibitem{Obe}
E.~Kirchberg and N.~C. Phillips.
\newblock \Cs-Algebras.
\newblock pp. 2115--2118.
\newblock Mathematisches Forschungsinstitut Oberwolfach, Report no. 37/2008.

\bibitem{KP00}
E.~Kirchberg and N.~C. Phillips.
\newblock Embedding of exact \Cs-algebras in the Cuntz algebra $\O2$.
\newblock J. reine angew. Math. 525 (2000), pp. 17--53.

\bibitem{KP23}
E.~Kirchberg and N.~C. Phillips.
\newblock Minimal dynamical systems on prime \Cs-algebras  (2023).
\newblock \urlprefix\url{https://arxiv.org/abs/2312.03980}, preprint, 42 pp.

\bibitem{Kis81}
A.~Kishimoto.
\newblock Outer automorphisms and reduced crossed products of simple
  \Cs-algebras.
\newblock Comm. Math. Phys. 81 (1981), pp. 429--435.

\bibitem{KM18}
B.~K. Kwa{\'s}niewski and R.~Meyer.
\newblock Aperiodicity, topological freeness and pure outerness: from group
  actions to Fell bundles.
\newblock Studia Math. 241 (2018), no.~3, pp. 257--302.

\bibitem{Lan95}
E.~C. Lance.
\newblock Hilbert \Cs-modules: a toolkit for operator algebraists, vol. 210.
\newblock Cambridge University Press (1995).

\bibitem{OS21}
N.~Ozawa and Y.~Suzuki.
\newblock On characterizations of amenable \Cs-dynamical systems and new
  examples.
\newblock Selecta Math. 27 (2021), no.~92.
\newblock 29 pp.

\bibitem{PS23}
M.~Pagliero and G.~Szab{\'o}.
\newblock Classification of equivariantly $\O2$-stable amenable actions on
  nuclear \Cs-algebras.
\newblock J. Funct. Anal. 288 (2025), no.~2.
\newblock Article 110683, 50 pp.

\bibitem{Phi00}
N.~C. Phillips.
\newblock A classification theorem for nuclear purely infinite simple
  \Cs-algebras.
\newblock Doc. Math. 5 (2000), pp. 49--114.

\bibitem{Pi97}
M.~V. Pimsner.
\newblock A class of \Cs-algebras generalizing both Cuntz-Krieger algebras and
  crossed product by $\Z$.
\newblock Fields Inst. Comm. 12 (1997), pp. 189--212.

\bibitem{RW98}
I.~Raeburn and D.~P. Williams.
\newblock Morita equivalence and continuous-trace \Cs-algebras, vol.~60.
\newblock Mathematical Surveys and Monographs, Amer. Math. Soc., Providence, RI
  (1998).

\bibitem{Ror92}
M.~R{\o}rdam.
\newblock On the structure of simple \Cs-algebras tensored with a UHF-algebra,
  II.
\newblock J. Funct. Anal. 107 (1992), no.~2, pp. 255--269.

\bibitem{Suz19}
Y.~Suzuki.
\newblock Simple equivariant \Cs-algebras whose full and reduced crossed
  products coincide.
\newblock J. Noncommut. Geom. 13 (2019), pp. 1577--1585.

\bibitem{Suz24}
Y.~Suzuki.
\newblock Every countable group admits amenable actions on stably finite simple
  \Cs-algebras.
\newblock To appear in Amer. J. Math.  (2024).
\newblock \urlprefix\url{http://arxiv.org/abs/arXiv:2204.04480}, 8 pp.

\bibitem{Sza18a}
G.~Szab{\'o}.
\newblock Strongly self-absorbing $\Cs$-dynamical systems.
\newblock Trans. Amer. Math. Soc. 370 (2018), no.~1, pp. 99--130.

\bibitem{Sza21}
G.~Szab{\'o}.
\newblock On a categorical framework for classifying $\Cs$-dynamics up to
  cocycle conjugacy.
\newblock J. Funct. Anal. 280 (2021), no.~8.
\newblock Article 108927, 66 pp.

\bibitem{TW07}
A.~Toms and W.~Winter.
\newblock Strongly self-absorbing \Cs-algebras.
\newblock Trans. Amer. Math. Soc. 359 (2007), no.~8, pp. 3999--4029.

\bibitem{HR98}
J.~v.~B.~Hjelmborg and M.~R{\o}rdam.
\newblock On Stability of \Cs-Algebras.
\newblock J. Funct. Anal. 155 (1998), no.~1, pp. 153--170.

\bibitem{Zha92}
S.~Zhang.
\newblock Certain \Cs-algebras with real rank zero and their corona and
  multiplier algebras. I.
\newblock Pacific J. Math.  (1992), no. 155, pp. 169--197.

\end{thebibliography}
\bibliographystyle{mybstnum.bst}

\end{document}